\documentclass[12pt,a4paper,reqno]{amsart}

\usepackage{amsmath,amssymb,amsthm}
\usepackage[english]{babel}
\usepackage[utf8]{inputenc}
\usepackage{tikz} 
\usepackage{float}
\usepackage[T1]{fontenc} 
\usepackage{mathtools}
\usepackage{mathrsfs}
\usepackage{upgreek}
 
\usepackage{graphicx}
\usepackage{epstopdf}
 
\usepackage{hyperref} 
\usepackage{cleveref}
\usepackage{autonum}

\usepackage[margin=0.9in]{geometry}
\newtheorem{theorem}{Theorem}[section]
\newtheorem{lemma}[theorem]{Lemma}

\newtheorem{proposition}[theorem]{Proposition}
\newtheorem{corollary}[theorem]{Corollary}

{\theoremstyle{remark} \newtheorem{remark}[theorem]{Remark}}
\numberwithin{equation}{section}

\newcommand{\C}{{\mathbb C}}

\newcommand{\sph}{{{\mathbb S}^{d-1}}}

\newcommand{\Z}{{\mathbb Z}}

\newcommand{\R}{{\mathbb R}}

\newcommand{\diam}{{\rm diam}}
\newcommand{\ddiv}{{\rm div}}

\newcommand{\paO}{{\partial\Omega}}
\newcommand{\re}{\operatorname{Re}}
\newcommand{\im}{\operatorname{Im}}

\makeatletter
\@namedef{subjclassname@2020}{\textup{2020} Mathematics Subject Classification}
\makeatother

\title[Poincar\'e inequalities for the logarithmic smoothness of the Gauss map]{Sharp Poincar\'e inequalities for the logarithmic smoothness of the Gauss map on hypersurfaces}

\date{\today}

\author[A. Mas]{Albert Mas}
\address{A. Mas \textsuperscript{1,2}
\newline
\textsuperscript{1}
Departament de Matem\`atiques,
Universitat Polit\`ecnica de Catalunya,
Campus Diagonal Bes\`os, Edifici A (EEBE), Av. Eduard Maristany 16, 08019
Barcelona, Spain
\newline
\textsuperscript{2}
Centre de Recerca Matem\`atica, Edifici C, Campus Bellaterra, 08193 Bellaterra, Spain}
\email{albert.mas.blesa@upc.edu}

\thanks{The author is supported by the Spanish grant PID2025-168310NB-I00 funded by MICIU/AEI/10.13039/501100011033 and by FSE+. This work is supported by the Spanish State Research Agency, through the Severo Ochoa and Mar\'ia de Maeztu Program for Centers and Units of Excellence in R\&D (CEX2020-001084-M)}

\begin{document}

\begin{abstract}
We establish sharp lower bounds for two geometric functionals which evaluate the global variation of the Gauss map on the boundary of bounded open sets $\Omega\subset\R^d$. These   functionals arise as the endpoint case $s=0$ of the Gagliardo seminorm in $W^{s,p}(\partial\Omega)$ of the normal vector field, for $p\in\{1,2\}$. We also show that balls are the unique minimizers. In addition, some applications to the fractional setting $s>0$ are given.
\end{abstract}

\maketitle

\section{Introduction}

Throughout this work we assume that $d\geq2$ is an integer. Given $\Omega \subset \mathbb{R}^d$ a bounded open set with Lipschitz boundary, let $\sigma$ denote the surface measure on $\paO$ and $\nu: \partial\Omega \to \mathbb{S}^{d-1}$ be the outward unit normal vector field on 
$\partial\Omega$ (the Gauss map), which is well defined $\sigma$-a.e. We want to investigate the relation between the geometric functionals 
\begin{equation}
\mathcal{E}_p(\Omega) := \int_{\partial\Omega}\!\int_{\partial\Omega}\frac{|\nu(x)-\nu(y)|^p}{|x-y|^{d-1}} \,d\sigma(x)\,d\sigma(y)\quad\text{for $p>0$}
\quad\text{and}\quad
\mathcal{P}(\Omega):=\sigma(\partial\Omega).
\end{equation}
The quantity $\mathcal{E}_p(\Omega)$ naturally arises as a limiting case in the context of fractional Sobolev spaces. Recall that, for a generic fractional parameter $s \in (0,1)$ and exponent $p \geq 1$, the standard Gagliardo seminorm of a function $f\in L^p(\paO)$ is given by
\begin{equation}\label{def:Gagl_s}
[f]_{W^{s,p}(\partial\Omega)}^p := \int_{\partial\Omega}\!\int_{\partial\Omega} \frac{|f(x) - f(y)|^p}{|x-y|^{d-1 + sp}} \,d\sigma(x)\,d\sigma(y).
\end{equation}
In this definition (which also makes sense for $p>0$), if we formally take the endpoint case $s=0$ and the $\mathbb{S}^{d-1}$-valued function $f=\nu$, we get $\mathcal{E}_p(\Omega)$.

In the flat setting, namely, $\paO=\R^{d-1}\times\{0\}\equiv \R^n$ with  $n=d-1$, it is well known 
that
\begin{equation}\label{MS_lim_s_0}
\lim_{s\downarrow 0} s[f]_{W^{s,p}(\R^n)}^p
=\tfrac{2}{p}|\mathbb{S}^{n-1}|\|f\|_{L^p(\R^n)}^p
\end{equation} 
for every compactly supported smooth function $f$; see \cite[Theorem 3]{MS}. This means that the Gagliardo seminorm blows up as $s\downarrow 0$ and a correction factor $s$ must be placed in front to recover the $L^p$ norm. However, in the setting of the present work, where $\paO$ is assumed to be bounded, the Gagliardo seminorm \eqref{def:Gagl_s} is actually finite for $s=0$ and every smooth function~$f$. Moreover, contrary to the flat case \eqref{MS_lim_s_0}, in general one does not recover the $L^p$ norm of $f$ taking $s=0$ in \eqref{def:Gagl_s}, even for $f=\nu$. In fact, as we will see,  
the quantity $\mathcal{E}_p(\Omega)$ measures the smoothness of the normal vector $\nu$, and thus of $\partial\Omega$, in a certain generalized sense. In particular, there exists a bounded open set $\Omega\subset\R^d$ with Lipschitz boundary, hence $\|\nu\|_{L^p(\paO)}^p=\mathcal{P}(\Omega)<+\infty$, for which $\mathcal{E}_p(\Omega)=+\infty$ for all $p>0$. 

The purpose of this work is to establish sharp inequalities between $\mathcal{E}_p(\Omega)$ and $\mathcal{P}(\Omega)$, and to identify the optimizers $\Omega\subset\R^d$; we will only cover the cases $p=1$ and $p=2$. The following is our main result in this regard.

\begin{theorem}\label{main thm intro}
Let $\Omega\subset\R^d$ be a bounded open set with Lipschitz boundary. Then,
\begin{equation}\label{ineq norm1}
\mathcal{E}_1(\Omega)
\geq|\sph|\mathcal{P}(\Omega)
\end{equation}
and
\begin{equation}\label{ineq norm1 main1}
\mathcal{E}_2(\Omega)\geq \frac{4 |\mathbb{S}^{d-2}|}{d-1} \mathcal{P}(\Omega),
\end{equation}
where, for a positive integer $n$, $|\mathbb{S}^{n-1}|$ denotes the surface area of the unit sphere $\mathbb{S}^{n-1}\subset\mathbb{R}^{n}$. Moreover, the equality in \eqref{ineq norm1} and in \eqref{ineq norm1 main1} holds if and only if 
$\Omega$ is a ball.
\end{theorem}

\begin{remark}
It is not clear how to get \eqref{ineq norm1} from \eqref{ineq norm1 main1}, or viceversa, since $|\nu(x)-\nu(y)|$ can be either smaller or larger than $1$ as $x$ and $y$ move along $\paO$. To compare both inequalities, one may look at the case of the unit ball $\mathbb{B}^d:=\{x\in\R^d:\,|x|<1\}$, where the equalities in \eqref{ineq norm1} and \eqref{ineq norm1 main1} hold, to see that 
$$\frac{\mathcal{E}_2(\mathbb{B}^d)}{\mathcal{E}_1(\mathbb{B}^d)}=\frac{4 |\mathbb{S}^{d-2}|}{(d-1)|\sph|}
=\frac{4\Gamma(\tfrac{d}{2})}{\sqrt{\pi}(d-1)\Gamma(\tfrac{d-1}{2})}
=\frac{2\Gamma(\tfrac{d}{2})}{\sqrt{\pi}\Gamma(\tfrac{d+1}{2})};$$ 
here $\Gamma$ denotes the Gamma function.
This quantity is strictly decreasing in $d$, equals $\frac{4}{\pi}$ for $d=2$, equals 1 for $d=3$, and converges to 0 as $d\uparrow+\infty$.
\end{remark}

Next, we describe the contents of the paper. In \Cref{s:log_smooth} we will relate 
$\mathcal{E}_2(\Omega)$ with the notion of {\em logarithmic smoothness}; see \cite{Dominguez2023}. This will allow us to interpret \Cref{main thm intro} as sharp logarithmic smoothness Poincar\'e-type inequalities in the setting of Gauss maps for bounded open sets in $\R^d$ with Lipschitz boundary. Then, in
\Cref{s:examples} we will provide some examples of Lipschitz domains $\Omega$ with finite/infinite energy~$\mathcal{E}_p(\Omega)$. This will show that the Lipschitz class is wide enough to contain both scenarios. Afterwards,
in \Cref{s:applications} we will give some applications of \Cref{main thm intro} to the setting of fractional smoothness. These are lower bounds for $[\nu]_{W^{{s},{p}}(\partial\Omega)}$ for certain values of ${s}\in(0,1)$ and 
${p}>1$ which, in some cases, are actually sharp. In particular, we will show that
\begin{equation}
\begin{split}
&\text{if $d=2$ then}\quad[\nu]_{W^{1-1/{p},{p}}(\paO)}\geq
2\pi\mathcal{P}(\Omega)^{2/{p}-1}
\quad\text{for all ${p}>1$ and}\\
&\text{if $d=3$ then}\quad[\nu]_{W^{1-2/{p},{p}}(\paO)}\geq
2\sqrt{\pi}\mathcal{P}(\Omega)^{2/{p}-1/2}
\quad\text{for all ${p}>2$.}
\end{split}
\end{equation} 
In both cases, the equality holds if and only if $\Omega$ is a ball; 
see \Cref{coro:sharp_Wsp}. A lower bound for a fractional analogue of the squared second fundamental form of $\paO$ will also be given; see \Cref{app:frac_curv}.
Finally, in \Cref{s monotonicity perimeter} we will prove \Cref{main thm intro}; at the beginning of the section we will describe the main ingredients.

Our results naturally fit within the broader study of lower bounds for curvature integrals, which play a key role in differential geometry and geometric analysis.  Notable classical results include the lower bounds on the total absolute curvature of closed curves obtained by Fenchel \cite{Fenchel1929} and Borsuk \cite{Borsuk1947} and extended to knots by F\'ary \cite{Fary1949} and Milnor \cite{Milnor1950}; the lower bounds for mean curvature energies on higher-dimensional submanifolds established by Willmore \cite{Willmore1968} and Chen \cite{Chen1971}; and the optimal relation between surface area and total mean curvature on convex domains due to Minkowski \cite{Minkowski1903}. A large body of literature emerged from these foundational results. In the present article, we contribute to this literature by establishing lower bounds for the logarithmic and fractional smoothness of the Gauss map on hypersurfaces. Since curvature reflects the variation of the Gauss map in a classical sense, our work approaches this field from a nonlocal perspective by quantifying this variation in a generalized sense.

\section{Logarithmic smoothness}\label{s:log_smooth}
The purpose of this section is to relate  
$\mathcal{E}_2(\Omega)$ to the notion of {logarithmic smoothness}. In particular, our aim is to show how \eqref{ineq norm1 main1} refers to a logarithmic smoothness Poincar\'e-type inequality in the setting of Gauss maps for bounded open sets in $\R^d$ with Lipschitz boundary. 
For simplicity, in the following we will only consider the case $d=2$, where  we will use the identification $\R^2\equiv\C$.

Let us start by introducing the functional
\begin{equation}
E_2(f):=\int_{-\pi}^\pi \!\int_{-\pi}^\pi\frac{|f(\theta)-f(\omega)|^2}{|e^{i\theta}-e^{i\omega}|}\,d\theta\,d\omega
\end{equation}
acting on $2\pi$-periodic functions $f:\R\to\C$. Note that the denominator of the integrand takes into account the distance between points $\theta,\omega\in(-\pi,\pi)$ in a $2\pi$-periodic fashion. Actually, $E_2(f)$ can be naturally rewritten as a functional acting on functions defined in $\mathbb{S}^1$. More precisely, if we associate to $f$ the function 
$F:\mathbb{S}^1\subset\C\to\C$ given by $F(x):=f(\theta)$ for $x=e^{i\theta}$, then
\begin{equation}
E_2(f)=\int_{\mathbb{S}^1}\!\int_{\mathbb{S}^1} \frac{|F(x)-F(y)|^2}{|x-y|}\,d\sigma(x)\,d\sigma(y).
\end{equation}

Assume that the $2\pi$-periodic function $f$ is such that
$f\in L^2(-\pi,\pi)$ and write it in Fourier series as $f(\theta)=\sum_{k\in\Z} f_k e^{ik\theta}$. Our first goal is to write $E_2(f)$ in terms of the Fourier coefficients $f_k$ in order to guess the smoothness ---that is, the decay of the Fourier coefficients--- that the condition $E_2(f)<+\infty$ imposes on $f$. First, using the change of variables $\theta=\omega+h$ in the integral on the $\theta$ variable and the $2\pi$-periodicity of $h\mapsto f(\omega+h)$ and $h\mapsto e^{ih}$, we see that
\begin{equation}
\begin{split}
E_2(f)&=\int_{-\pi}^\pi\!\int_{-\pi-\omega}^{\pi-\omega}
\frac{|f(\omega+h)-f(\omega)|^2}{|e^{ih}-1|}\,dh\,d\omega
=\int_{-\pi}^\pi\!\int_{-\pi}^{\pi}
\frac{|f(\omega+h)-f(\omega)|^2}{|e^{ih}-1|}\,dh\,d\omega\\
&=\int_{-\pi}^\pi\frac{1}{|e^{ih}-1|}\!\int_{-\pi}^{\pi}
|f(\omega+h)-f(\omega)|^2\,d\omega\,dh.
\end{split}
\end{equation}
Since $f(\omega+h)-f(\omega)=\sum_{k\neq0}f_k (e^{ikh}-1)e^{ik\omega}$, using Parseval's identity in the integral on the $\omega$ variable and monotone convergence to interchange the sum and the integral, we get
\begin{equation}\label{E2_Fourier:eq1}
\begin{split}
E_2(f)
=2\pi\sum_{k\neq0}|f_k|^2
\int_{-\pi}^\pi\frac{|e^{i|k|h}-1|^2}{|e^{ih}-1|}\,dh
=4\pi\sum_{k\neq0}|f_k|^2
\int_{0}^\pi\frac{|e^{i|k|h}-1|^2}{|e^{ih}-1|}\,dh,
\end{split}
\end{equation}
where we also used in the last equality above that the integrand is an even function.
For $k\in\Z$, let us denote by $I_k$ the last integral above, hence 
$E_2(f)=4\pi\sum_{k\neq0}|f_k|^2 I_k$.
Note that
\begin{equation}
\begin{split}
|e^{ih}-1|=|e^{i\frac{h}{2}}-e^{-i\frac{h}{2}}|=2|\sin(\tfrac{h}{2})|
=2\sin(\tfrac{h}{2})
\end{split}
\end{equation}
for all $h\in(0,\pi)$, and that
\begin{equation}
\begin{split}
|e^{i|k|h}-1|^2-|e^{i(|k|-1)h}-1|^2
&=-e^{i|k|h}-e^{-i|k|h}+e^{i(|k|-1)h}+e^{-i(|k|-1)h}\\
&=-(e^{i(|k|-\frac{1}{2})h}-e^{-i(|k|-\frac{1}{2})h})(e^{i\frac{h}{2}}-e^{-i\frac{h}{2}})\\
&=4\sin\big((|k|-\tfrac{1}{2})h\big)\sin(\tfrac{h}{2}).
\end{split}
\end{equation}
With these two identities in hand, and since $I_0=0$, for every $k\neq0$ we see that
\begin{equation}\label{E2_Fourier:eq2}
\begin{split}
I_k&=I_{|k|}=\sum_{l=1}^{|k|}(I_l-I_{l-1})
=\sum_{l=1}^{|k|}\int_{0}^\pi\frac{|e^{ilh}-1|^2-|e^{i(l-1)h}-1|^2}{|e^{ih}-1|}\,dh\\
&=2\sum_{l=1}^{|k|}\int_{0}^\pi 
\sin\big((l-\tfrac{1}{2})h\big)\,dh
=4\sum_{l=1}^{|k|}\frac{1}{2l-1}.
\end{split}
\end{equation}
Finally, plugging \eqref{E2_Fourier:eq2} in \eqref{E2_Fourier:eq1}, we conclude that 
\begin{equation}\label{Poincare-Besov_interval:eq2}
E_2(f)=16\pi\sum_{k\neq0}\lambda_k|f_k|^2,
\quad\text{where }
\lambda_k:=\sum_{l=1}^{|k|}\frac{1}{2l-1}.
\end{equation}

Viewing $\lambda_k$ as a Riemann sum, it is not hard to show that there exists $M>1$ such that
$\frac{1}{\lambda_k}(1+\ln|k|)\in(\tfrac{1}{M},M)$ for all $k\neq0$, which means that 
\begin{equation}\label{Poincare-Besov_interval:eq5}
E_2(f)\simeq\sum_{k\neq0}(1+\ln|k|)|f_k|^2.
\end{equation}
Therefore, the condition $E_2(f)<+\infty$ is equivalent to  
$\sum_{k\neq0}|f_k|^2\ln|k|<+\infty$, which is the so-called  
Kolmogorov-Seliverstov-Plessner condition for almost everywhere convergence of Fourier series; see \cite{Kolmogorov1926,Plessner1926}. Actually, thanks to \eqref{Poincare-Besov_interval:eq5}, we know that $E_2(f)<+\infty$ if and only if $f\in B^{0,1/2}_{2,2}(\mathbb{T})$; see \cite[Theorem 4.22]{Dominguez2023}. The latter is a Besov space on the torus having zero classical smoothness and an additional logarithmic smoothness with exponent $\frac{1}{2}$.

The explicit expression \eqref{Poincare-Besov_interval:eq2} naturally leads to a sharp Poincar\'e-type inequality for $E_2(f)$ and functions $f$ with zero mean.
Clearly, 
$\lambda_k>\lambda_1=\lambda_{-1}=1$ for all $|k|>1$. Therefore, if $f_0:=\frac{1}{2\pi}\int_{-\pi}^\pi f=0$  then Parseval's identity and \eqref{Poincare-Besov_interval:eq2} yield
\begin{equation}
\begin{split}
\|f\|_{L^2(-\pi,\pi)}^2=2\pi\sum_{k\neq0}|f_k|^2
\leq2\pi\sum_{k\neq0}\lambda_k|f_k|^2
=\frac{1}{8}E_2(f),
\end{split}
\end{equation}
and the equality holds if and only if $f_k=f_0=0$ for all $|k|\geq2$. This shows, on the one hand, that the logarithmic smoothness Poincar\'e-type inequality 
\begin{equation}\label{Poincare-Besov_interval:eq1}
\begin{split}
\|f\|_{L^2(-\pi,\pi)}^2\leq\frac{1}{8}E_2(f)
\quad\text{holds for all $f\in L^2(-\pi,\pi)$ such that ${\textstyle \int_{-\pi}^\pi f=0}$}
\end{split}
\end{equation}
and, on the other hand, that the constant $\frac{1}{8}$ is optimal in \eqref{Poincare-Besov_interval:eq1}
and the optimizers are 
\begin{equation}\label{Poincare-Besov_interval:eq1_opt}
f(\theta)=f_{1}e^{i\theta}+f_{-1}e^{-i\theta}\quad\text{with $f_1,f_{-1}\in\C$.}
\end{equation}

Let us now investigate \eqref{Poincare-Besov_interval:eq1} for the class of $\C$-valued functions $f$ given by the Gauss map on the boundary of domains 
$\Omega\subset\C$. To simplify the exposition, let us restrict  ourselves to the case where $\mathcal{P}(\Omega)=2\pi$ and  $\partial\Omega$ is a bi-Lipschitz Jordan curve. By this we mean that there exists a bijective function
$\Gamma:\mathbb{S}^1\to\partial\Omega$ such that 
\begin{equation}\label{Poincare-Besov_interval:bilip}
\tfrac{1}{C}|x-y|\leq|\Gamma(x)-\Gamma(y)|\leq C|x-y|\quad\text{for some $C>0$ and all $x,y\in\mathbb{S}^1$.}
\end{equation}
Next, set 
$\gamma(\theta):=\Gamma(e^{i\theta})$ for all $\theta\in\R$, hence $\gamma:\R\to\paO\subset\C$ is a $2\pi$-periodic Lipschitz function. Assume that $|\gamma'|=1$ a.e. in $\R$, which means that $\gamma$ restricted to $(-\pi,\pi]$ is an arc length parameterization of $\paO$. Then, since 
$\int_{-\pi}^\pi\gamma'=\gamma(\pi)-\gamma(-\pi)=0$, 
$\gamma'$ is an admissible competitor in \eqref{Poincare-Besov_interval:eq1}, which gives
\begin{equation}\label{Poincare-Besov_interval:eq3}
E_2(\gamma')\geq 8\|\gamma'\|_{L^2(-\pi,\pi)}^2
=16\pi
=8\mathcal{P}(\Omega).
\end{equation}

We will now address the case of equality in \eqref{Poincare-Besov_interval:eq3} in terms of the geometry of $\paO$. 
First, let us assume that $\partial\Omega$ is a cicle of radius $1$. In this case, we can take
$\gamma(\theta)=e^{i\theta}+B$ for some $B\in\C$ (the center of the cicle) and all $\theta\in\R$. Then, 
$\gamma'(\theta)=ie^{i\theta}$ is of the form \eqref{Poincare-Besov_interval:eq1_opt} and, therefore, equality holds in \eqref{Poincare-Besov_interval:eq3}. 
Next, we deal with the reverse implication; let us assume now that equality holds in \eqref{Poincare-Besov_interval:eq3}. Then, from \eqref{Poincare-Besov_interval:eq1_opt} we deduce that
$\gamma'(\theta)=c_{1}e^{i\theta}+c_{-1}e^{-i\theta}$ for some $c_1,c_{-1}\in\C$ and a.e. $\theta\in\R$; this in particular yields that, actually, $\gamma$ is a smooth function. But recall that $|\gamma'|=1$ by assumption, which leads to 
\begin{equation}\label{Poincare-Besov_interval:eq4}
|c_1|^2+|c_{-1}|^2+2\re(c_1\overline{c_{-1}}e^{2i\theta})=1
\quad\text{for all $\theta\in\R$.}
\end{equation}
From this, if we take $\theta=0$ and $\theta=\frac{\pi}{2}$ we deduce that
\begin{equation}
2\re(c_1\overline{c_{-1}})=1-|c_1|^2-|c_{-1}|^2
=-2\re(c_1\overline{c_{-1}}),
\end{equation}
which gives $\re(c_1\overline{c_{-1}})=0$. Similarly, taking $\theta=\frac{\pi}{4}$ and $\theta=-\frac{\pi}{4}$ we get
\begin{equation}
2\re(c_1\overline{c_{-1}}i)=1-|c_1|^2-|c_{-1}|^2
=-2\re(c_1\overline{c_{-1}}i),
\end{equation}
which gives $\im(c_1\overline{c_{-1}})
=-\re(c_1\overline{c_{-1}}i)=0$. Therefore, $c_1\overline{c_{-1}}=0$ and, in view of \eqref{Poincare-Besov_interval:eq4}, this leads to either $c_1=0$ and 
$|c_{-1}|=1$ or $c_{-1}=0$ and $|c_1|=1$. As a consequence, 
if the equality holds in \eqref{Poincare-Besov_interval:eq3} then $\gamma'(\theta)=ce^{\pm i\theta}$ for some $c\in\C$ with $|c|=1$. The latter gives 
$\gamma(\theta)=Ae^{\pm i\theta}+B$ for some $A,B\in\C$ with $|A|=1$ and all $\theta\in\R$, which means that $\partial\Omega$ is a cicle of radius $1$ centered at $B$. 

As a conclusion of the previous discussion, we know that $E_2(\gamma')\geq8\|\gamma'\|_{L^2(-\pi,\pi)}^2$, that the equality holds if and only if $\Omega$ is a disk of radius $1$, and that  the quantity $E_2(\gamma')$ encapsulates the logarithmic smoothness of $\gamma'$; recall \eqref{Poincare-Besov_interval:eq5}. In addition, since $|\gamma'|=1$,
we also know that $\nu(\gamma)=\pm i\gamma'$ a.e. in $\R$, where the ``$\pm$'' accounts for the clockwise/counterclockwise orientation of the parameterization of $\partial\Omega$ through $\gamma$. This last observation allows us to rewrite $E_2(\gamma')$ in terms of the parameterization $\gamma$ and the Gauss map $\nu$ as
\begin{equation}\label{Poincare-Besov_interval:eq6}
\begin{split}
E_2(\gamma')
=\int_{-\pi}^\pi \!\int_{-\pi}^\pi\frac{|\gamma'(\theta)-\gamma'(\omega)|^2}{|e^{i\theta}-e^{i\omega}|}\,d\theta\,d\omega
=\int_{-\pi}^\pi \!\int_{-\pi}^\pi\frac{|\nu(\gamma(\theta))-\nu(\gamma(\omega))|^2}{|e^{i\theta}-e^{i\omega}|}\,d\theta\,d\omega.
\end{split}
\end{equation}
At this point one can easily compare $E_2(\gamma')$ with $\mathcal{E}_2(\Omega)$.
More precisely, since $|e^{i\theta}-e^{i\omega}|\simeq|\Gamma(e^{i\theta})-\Gamma(e^{i\omega})|
=|\gamma(\theta)-\gamma(\omega)|$ 
by the assumption \eqref{Poincare-Besov_interval:bilip}, we see that 
\begin{equation}\label{Poincare-Besov_interval:eq6}
\begin{split}
E_2(\gamma')\simeq\int_{-\pi}^\pi \!\int_{-\pi}^\pi\frac{|\nu(\gamma(\theta))-\nu(\gamma(\omega))|^2}{|\gamma(\theta)-\gamma(\omega)|}\,d\theta\,d\omega
=\int_{\paO}\!\int_{\paO} \frac{|\nu(x)-\nu(y)|^2}{|x-y|}\,d\sigma(x)\,d\sigma(y)
=\mathcal{E}_2(\Omega)
\end{split}
\end{equation}
(recall that here $d=2$),
meaning that $E_2(\gamma')<+\infty$ if and only if 
$\mathcal{E}_2(\Omega)<+\infty$. 
In view of these last comments,  
$\mathcal{E}_2(\Omega)$ quantifies the logarithmic smoothness of the Gauss map on~$\partial\Omega$. Furthermore, since $\|\gamma'\|_{L^2(-\pi,\pi)}^2=\mathcal{P}(\Omega)$, \eqref{ineq norm1 main1} is an analogue of the sharp Poincar\'e-type inequality 
\eqref{Poincare-Besov_interval:eq1} (from which \eqref{Poincare-Besov_interval:eq3} follows) in the setting of Gauss maps for bounded open sets in $\R^d$ with Lipschitz boundary.

\section{On the condition $\mathcal{E}_p(\Omega)<+\infty$ for bounded Lipschitz domains}\label{s:examples}

It is clear that if $\Omega\subset\R^d$ is a bounded domain with $C^{1,\alpha}$ boundary for some $0<\alpha<1$ then 
$\mathcal{E}_p(\Omega)<+\infty$ for all $p>0$. Indeed, 
since $\nu\in C^{\alpha}(\partial\Omega,\R^d)$ (see \cite[(2.74)]{HMT2007}, for example), there exists $C>0$ such that 
$|\nu(x)-\nu(y)|\leq C|x-y|^\alpha$ for all $x,y\in\paO$. Then,
\begin{equation}\label{Ep_finite_holder}
\int_{\paO}\frac{|\nu(x)-\nu(y)|^p}{|x-y|^{d-1}}\,d\sigma(x)\leq C\int_{\paO}\frac{1}{|x-y|^{d-1-\alpha p}}\,d\sigma(x)\leq C'
\quad\text{for all $y\in\paO$,}
\end{equation}
where $C'>0$ is some constant independent of $y\in\paO$; the second inequality in \eqref{Ep_finite_holder} follows from the fact that $\paO$ is a compact $(d-1)$-dimensional Ahlfors regular set, that $d-1-\alpha p<d-1$, and a standard argument based on integration in dyadic annuli centered at $y$. Using \eqref{Ep_finite_holder} we conclude that 
$\mathcal{E}_p(\Omega)\leq C'\mathcal{P}(\Omega)$, thus $\mathcal{E}_p(\Omega)<+\infty$ for all $p>0$.

Although the previous argument fails if we only assume Lipschitz regularity of~$\paO$ ---since now $\nu$ may have jumps along the boundary---, one can still find domains $\Omega$ with corners for which $\mathcal{E}_p(\Omega)$ is finite. To illustrate this, let us show that $\mathcal{E}_p(\Omega)<+\infty$ for the simple case of the square $\Omega:=(0,1)^2\subset\R^2$ ---similar (but more tedious) computations can be carried out for simple domains with corners in $\R^d$ for $d\geq3$. By symmetry, note that
\begin{equation}
\mathcal{E}_p(\Omega) = 4\int_{0}^1\!\int_{0}^1\frac{|(0,1)-(0,-1)|^p}{\sqrt{1+|x-y|^2}} \,dx\,dy+8\int_{0}^1\!\int_{0}^1\frac{|(-1,0)-(0,-1)|^p}{\sqrt{x^2+y^2}} \,dx\,dy.
\end{equation}
The first double integral is obviously finite, and the second one can be bounded by
\begin{equation}
\begin{split}
2^{p/2}\int_{0}^1\!\int_{0}^1\frac{dx\,dy}{\sqrt{x^2+y^2}}
\leq2^{p/2}\iint_{\{x^2+y^2<2\}}\frac{dx\,dy}{\sqrt{x^2+y^2}}
=2^{(p+3)/2}\pi.
\end{split}
\end{equation}

As we have seen, the square $\Omega:=(0,1)^2\subset\R^2$ is an example of a bounded domain with Lipschitz boundary, but not $C^1$, for which $\mathcal{E}_p(\Omega)$ is finite for all $p>0$. A natural question, then, is whether 
$\mathcal{E}_p(\Omega)$ is always finite on bounded Lipschitz domains $\Omega\subset\R^d$. In view of the logarithmic smoothness discussed in \Cref{s:log_smooth}, one may guess that the answer is negative. The purpose of this section is to provide an explicit example. More precisely, given $d\geq2$, we will construct a bounded domain $\Omega\subset\R^d$ with Lipschitz boundary  such that 
$\mathcal{E}_p(\Omega)=+\infty$ for all $p>0$. 

In the following, for a point $x=(x_1,\ldots,x_d)\in\R^d$ we will use the notation 
\begin{equation}
x':=(x_1,\ldots,x_{d-1})\quad\text{and, for $d\geq3$, }\quad x'':=(x_2,\ldots,x_{d-1}).
\end{equation} 
Given $f\in L^\infty(0,1)$ with $\|f\|_{L^\infty(0,1)}=1$, set
$\psi(t):=f(|t|)$ for $t\in (-1,1)$. Next, let 
$\phi:(-1,1)^{d-1}\to\R$ be defined by 
$$\phi(x'):=\int_0^{x_1}\psi(t)\,dt;$$
since $|\phi(x')-\phi(y')|\leq\|f\|_{L^\infty(0,1)}|x_1-y_1|\leq|x'-y'|$ for all $x',y'\in (-1,1)^{d-1}$, $\phi$ is a Lipschitz function.
Finally, let $\Omega \subset \mathbb{R}^d$ be a bounded domain with Lipschitz boundary such that
$$\Omega\cap(-1,1)^{d}=\{x=(x',x_d)\in(-1,1)^{d}:\,x_d>\phi(x')\}.$$

For $\sigma$-a.e. $x\in\partial\Omega\cap(-1,1)^{d}$, the outward unit normal vector and surface measure write, respectively, as 
\begin{equation}\label{eq:bound_to_local_eq1}
\nu(x) = \frac{(\psi(x_1),0,\ldots 0, -1)}{\sqrt{1 + \psi(x_1)^2}},\quad
d\sigma(x)=\sqrt{1 + \psi(x_1)^2}\,dx'.
\end{equation}
Note that
\begin{equation}
\frac{b}{\sqrt{1+b^2}}-\frac{a}{\sqrt{1+a^2}}
=\int_a^b\frac{dt}{(1+t^2)^{3/2}}\geq\frac{b-a}{2^{3/2}}
\quad\text{for all $-1\leq a\leq b\leq 1$.}
\end{equation}
If we take 
$a=\min\{\psi(x_1),\psi(y_1)\}$ and $b=\max\{\psi(x_1),\psi(y_1)\}$ for any given $x,y\in\partial\Omega\cap(-1,1)^{d}$, then $-1\leq a\leq b\leq 1$ because $\|f\|_{L^\infty(0,1)}=1$ and, therefore,
\begin{equation}
\begin{split}
|\nu(x)-\nu(y)|^p&\geq
\bigg|\frac{\psi(x_1)}{\sqrt{1 + \psi(x_1)^2}}
-\frac{\psi(y_1)}{\sqrt{1 + \psi(y_1)^2}}\bigg|^p
\geq \frac{1}{2^{3p/2}}|\psi(x_1)-\psi(y_1)|^p.
\end{split}
\end{equation}
Moreover,
$|x-y|^{2}=|x'-y'|^2+{\textstyle |\int_{y_1}^{x_1}\psi}|^2
\leq2|x_1-y_1|^2+|x''-y''|^2
\leq 2|x'-y'|^2$.
With these estimates in hand, if we set 
\begin{equation}
\begin{split}
Q&:=\{x'\in\R^{d-1}:\,0<x_1<1,\,|x''|<\tfrac{1}{2}\},\\
Q_x&:=\{y'\in\R^{d-1}:\,0<y_1<1\text{ and }|y''-x''|<\tfrac{1}{10}\}\quad\text{for $x'\in Q$}
\end{split}
\end{equation}
then, recalling the second equality in \eqref{eq:bound_to_local_eq1}, we can estimate
\begin{equation}\label{ex1:eq1}
\begin{split}
\mathcal{E}_p(\Omega) &\geq \int_{\partial\Omega\cap\{x\in(-1,1)^d:\,x'\in Q\}}\!\int_{\partial\Omega\cap\{y\in(-1,1)^d:\,y'\in Q_x\}}\frac{|\nu(x)-\nu(y)|^p}{|x-y|^{d-1}} \,d\sigma(y)\,d\sigma(x)\\
&\geq C \int_{\{x'\in Q\}}\!\int_{\{y'\in Q_x\}}
\frac{|\psi(x_1)-\psi(y_1)|^p}{|x'-y'|^{d-1}}\,dy'\,dx'
\end{split}
\end{equation}
for some constant $C>0$ depending only on $p$.
For $d=2$, this writes as
\begin{equation}\label{ex1:eq2}
\begin{split}
\mathcal{E}_p(\Omega) 
\geq C \int_{0}^{1}\!
\int_{0}^{1}
\frac{|\psi(x_1)-\psi(y_1)|^p}{|x_1-y_1|}
\,dy_1\,dx_1.
\end{split}
\end{equation}
For $d\geq3$, we proceed as follows.
Given $x'\in Q$, by a change of variables to spherical coordinates in $\R^{d-2}$ centered at $x''$ we see that
\begin{equation}
\begin{split}
\int_{\{y'\in Q_x\}}&\frac{|\psi(x_1)-\psi(y_1)|^p}
{|x'-y'|^{d-1}}\,dy'\\
&=\int_{0}^{1}\frac{|\psi(x_1)-\psi(y_1)|^p}{|x_1-y_1|^{d-1}}\!
\int_{\{|y''-x''|<\frac{1}{10}\}}
\Big(1+\frac{|x''-y''|^2}{|x_1-y_1|^2}\Big)^{(1-d)/2}\,dy''\,dy_1\\
&=\int_{0}^{1}\frac{|\psi(x_1)-\psi(y_1)|^p}{|x_1-y_1|}\!
\int_{\{|z''|<\frac{1}{10|x_1-y_1|}\}}
(1+|z''|^2)^{(1-d)/2}\,dz''\,dy_1\\
&=\int_{0}^{1}\frac{|\psi(x_1)-\psi(y_1)|^p}{|x_1-y_1|}|\mathbb{S}^{d-3}|
\int_{0}^{\frac{1}{10|x_1-y_1|}}
\frac{r^{d-3}}{(1+r^2)^{(d-1)/2}}\,dr\,dy_1\\
&\geq \int_{0}^{1}\frac{|\psi(x_1)-\psi(y_1)|^p}{|x_1-y_1|}
|\mathbb{S}^{d-3}|\int_{0}^{\frac{1}{10}}
\frac{r^{d-3}}{(1+r^2)^{(d-1)/2}}\,dr\,dy_1\\
&= C\int_{0}^{1}\frac{|\psi(x_1)-\psi(y_1)|^p}{|x_1-y_1|}\,dy_1,
\end{split}
\end{equation}
for some constant $C>0$ depending only on $d$. Plugging this estimate into \eqref{ex1:eq1} and integrating with respect to $x''$ (recall that $|x''|<\tfrac{1}{2}$ for all $x'\in Q$) leads to
\begin{equation}
\begin{split}
\mathcal{E}_p(\Omega) 
\geq C \int_{\{x'\in Q\}} \!\int_{0}^{1}
\frac{ |\psi(x_1)-\psi(y_1)|^p}{|x_1-y_1|}\,dy_1\,dx'
=C \int_{0}^1\!\int_{0}^{1} \frac{|\psi(x_1)-\psi(y_1)|^p}{|x_1-y_1|}\,dy_1\,dx_1,
\end{split}
\end{equation}
where $C>0$ is some constant depending only on $d$ and 
$p$. This agrees with the estimate that we have shown in \eqref{ex1:eq2} for the case $d=2$. Now, since 
$x_1,y_1\in (0,1)$, we have that $\psi(x_1)-\psi(y_1)=f(x_1)-f(y_1)$. Thus, choosing $f\in L^\infty(0,1)$ as the one in \Cref{lm:jump_fct} below, we conclude that $\Omega$ is a bounded domain with Lipschitz boundary such that 
$\mathcal{E}_p(\Omega)=+\infty$ for all $p>0$ ---a plot of 
$\partial\Omega\cap(-1,1)^d$ for $d=3$ can be found in \Cref{Fig:Plot_inf_energy}.

\begin{figure}[h!]
	\centering
	\includegraphics[width=0.5\linewidth]{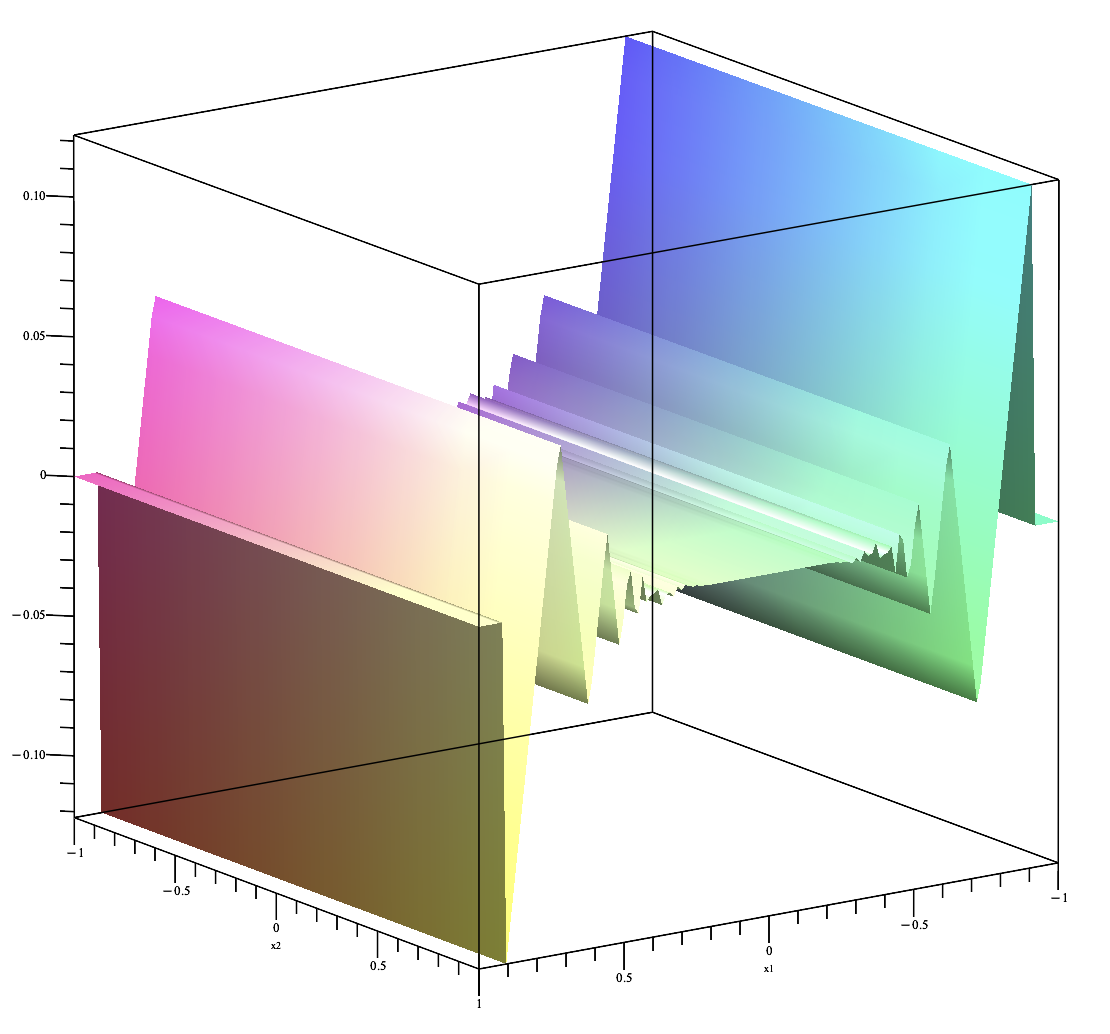}
	\caption{Plot of the graph of $\phi$ for $d=3$ with $f$ as in \Cref{lm:jump_fct} ---recall that 
	$\partial\Omega\cap(-1,1)^{3}=\{(x',\phi(x'))\in\R^3:\,x'=(x_1,x_2)\in(-1,1)^{2}\}$. Although in the plot the graph of $\phi$ looks smooth for $x_1$ close to 0, it actually oscillates infinitely many times as $x_1\to0$ with slopes $\pm1$ given by $f$.}
	\label{Fig:Plot_inf_energy}
\end{figure}

\begin{lemma}\label{lm:jump_fct}
For $k=3,4,5,\ldots$ set $x_k:=\tfrac{1}{\ln k}$ and $I_k:=(x_{k+1},x_k)$. Then, $\{I_k\}_{k\geq3}$ is a sequence of disjoint intervals such that $\cup_{k\geq3}I_k\subset(0,1)$.
On the interval $(0,1)$, define the function $f:=\sum_{k\geq 3}(-1)^k\chi_{I_k}$. Then $\|f\|_{L^\infty(0,1)}=1$ and 
\begin{equation}
\begin{split}
\int_{0}^{1}\!\int_{0}^{1}\frac{|f(x)-f(y)|^p}{|x-y|}\,dy\,dx
=+\infty\quad\text{for all $p>0.$}
\end{split}
\end{equation}
\end{lemma}

\begin{proof}
The fact that $\|f\|_{L^\infty(0,1)}=1$ is clear, since 
$\{I_k\}_{k\geq3}$ is a sequence of disjoint intervals.

Next, we will show that the double integral is infinite. Note that, if $j\geq 2$, $k\geq j+1$, $x\in I_{2j}$, and $y\in I_{2k+1}$, then 
$|f(x)-f(y)|^p=2^p$ and 
$\tfrac{1}{|x-y|}\geq\tfrac{1}{x_{2j}-x_{2k+2}}$. Thus,
\begin{equation}\label{char_eq1}
\begin{split}
\int_{0}^{1}\!\int_{0}^{1}\frac{|f(x)-f(y)|^p}{|x-y|}\,dy\,dx
&\geq2^p\sum_{j\geq2}\sum_{k=j+1}^{2j}\int_{I_{2j}}\int_{I_{2k+1}}\frac{dy\,dx}{|x-y|}\\
&\geq2^p\sum_{j\geq2}\sum_{k=j+1}^{2j}\frac{(x_{2j}-x_{2j+1})(x_{2k+1}-x_{2k+2})}{x_{2j}-x_{2k+2}}.
\end{split}
\end{equation}
On the one hand, since $\tfrac{d}{dt}(\tfrac{1}{\ln t})=-\tfrac{1}{t\ln^2 t}$, 
we see that
\begin{equation}\label{char_eq2}
x_{2j}-x_{2j+1}=\frac{1}{\ln (2j)}-\frac{1}{\ln (2j+1)}
=\int_{2j}^{2j+1}\frac{dt}{t\ln^2 t}
\geq\frac{1}{(2j+1)\ln^2(2j+1)}
\end{equation}
for all $j\geq 2$ and, similarly,
\begin{equation}\label{char_eq3}
x_{2k+1}-x_{2k+2}
\geq\frac{1}{(2k+2)\ln^2(2k+2)}
\geq\frac{1}{(4j+2)\ln^2(4j+2)}
\end{equation}
for all $j\geq2$ and $j+1\leq k\leq 2j$.
On the other hand, using that $\ln(1+t)\leq t$ for all $t\geq0$, we have
\begin{equation}\label{char_eq4}
\begin{split}
x_{2j}-x_{2k+2}&=\frac{1}{\ln (2j)}-\frac{1}{\ln (2k+2)}
=\frac{\ln (2k+2)-\ln (2j)}{\ln (2j)\ln (2k+2)}\\
&=\frac{\ln\big(1+\tfrac{k-j+1}{j}\big)}{\ln (2j)\ln (2k+2)}
\leq\frac{k-j+1}{j\ln (2j)\ln (2k+2)}
\leq\frac{k-j+1}{j\ln ^2(2j)}
\end{split}
\end{equation}
for all $j\geq2$ and $j+1\leq k$. Combining \eqref{char_eq2}, \eqref{char_eq3}, and \eqref{char_eq4} we deduce that 
\begin{equation}\label{char_eq5}
\begin{split}
\frac{(x_{2j}-x_{2j+1})(x_{2k+1}-x_{2k+2})}{x_{2j}-x_{2k+2}}\geq \frac{j\ln ^2(2j)}{(4j+2)^2\ln^4(4j+2)(k-j+1)}
\end{split}
\end{equation}
for all $j\geq2$ and $j+1\leq k\leq 2j$. Since
\begin{equation}
\lim_{j\uparrow+\infty}\frac{j}{4j+2}=\frac{1}{4}
\quad\text{and}\quad 
\lim_{j\uparrow+\infty}\frac{\ln(2j)}{\ln(4j+2)}=
\lim_{j\uparrow+\infty}\frac{\ln j}{\ln(4j+2)}=1,
\end{equation}
from \eqref{char_eq5} we deduce that there exists $j_0\geq2$ such that 
\begin{equation}\label{char_eq6}
\begin{split}
\frac{(x_{2j}-x_{2j+1})(x_{2k+1}-x_{2k+2})}{x_{2j}-x_{2k+2}}\geq \frac{1}{20j\ln^2j(k-j+1)}
\end{split}
\end{equation}
for all $j\geq j_0$ and $j+1\leq k\leq 2j$. Now, plugging this estimate in \eqref{char_eq1} leads to
\begin{equation}\label{char_eq7}
\begin{split}
\int_{0}^{1}\!\int_{0}^{1}\frac{|f(x)-f(y)|^p}{|x-y|}\,dy\,dx
&\geq2^p\sum_{j\geq j_0}\sum_{k=j+1}^{2j}\frac{(x_{2j}-x_{2j+1})(x_{2k+1}-x_{2k+2})}{x_{2j}-x_{2k+2}}\\
&\geq\frac{2^p}{20}\sum_{j\geq j_0}\frac{1}{j\ln^2j}\sum_{k=j+1}^{2j}\frac{1}{k-j+1}
=\frac{2^p}{20}\sum_{j\geq j_0}\frac{1}{j\ln^2j}
\sum_{n=2}^{j+1}\frac{1}{n}.
\end{split}
\end{equation}
Viewing the last sum above as an upper Riemann sum, we note that
\begin{equation}
\sum_{n=2}^{j+1}\frac{1}{n}\geq\int_2^{j}\frac{dt}{t}=\ln j-\ln 2
\geq \frac{1}{2}\ln j
\end{equation}
for all $j\geq 4$. Finally, inserting this simple estimate in \eqref{char_eq7} yields
\begin{equation}\label{char_eq8}
\begin{split}
\int_{0}^{1}\!\int_{0}^{1}\frac{|f(x)-f(y)|^p}{|x-y|}\,dy\,dx
&\geq\frac{2^p}{40}\sum_{j\geq j_0}\frac{1}{j\ln j}
\geq\frac{2^p}{40}\int_{j_0}^{+\infty}\frac{dt}{t\ln t}\\
&=\frac{2^p}{40}\lim_{T\uparrow+\infty}\big(\ln(\ln T)-\ln(\ln j_0)\big)=+\infty,
\end{split}
\end{equation}
which proves the lemma.
\end{proof}

\section{Applications to fractional smoothness}\label{s:applications}
The purpose of this section is to show how \Cref{main thm intro} easily yields lower bounds for $[\nu]_{W^{{s},{p}}(\partial\Omega)}$ for certain values of  ${s}\in(0,1)$ and 
${p}>1$. We will explore two different approaches, assuming in both cases a lower bound for 
$\mathcal{E}_p(\Omega)$, for some $p>0$, as a starting point. Then, we will apply our conclusions to the cases 
$p=1$ and $p=2$, for which we know such lower bounds with sharp constants thanks to \Cref{main thm intro}.

Let $p>0$. Assume that there exists  
$C_0>0$ such that
\begin{equation}\label{conj:ineq_all_p}
\mathcal{E}_p(\Omega)\geq C_{0}\mathcal{P}(\Omega)
\end{equation}
for every bounded open set $\Omega\subset\R^d$ with Lipschitz boundary. The following is a first simple argument that leads to (nonsharp) lower bounds for 
$[\nu]_{W^{{s},{p}}(\partial\Omega)}$ for all $s\in(0,1)$.
Given 
$\Omega\subset\R^d$ a bounded open set with Lipschitz boundary, set $\Omega':=\diam(\Omega)^{-1}\Omega$. Then
$\diam(\Omega')=1$, which implies that $|x'-y'|\leq1$ for all $x',y'\in\partial\Omega'$. Therefore, 
\begin{equation}\label{eq:s0_to_s>0_1}
\frac{1}{|x'-y'|^{d-1+sp}}\geq\frac{1}{|x'-y'|^{d-1}}\quad\text{for all $s\in(0,1)$ and $x',y'\in\partial\Omega'$ with $x'\neq y'$.}
\end{equation} 
Using this estimate together with \eqref{conj:ineq_all_p} on 
$\Omega'$ (and, to avoid confusions, denoting by $\nu'$ the outward unit normal vector field on $\partial\Omega'$), by a scaling argument we deduce that
\begin{equation}\label{frac_s:almost_coro1}
\begin{split}
[\nu]_{W^{s,p}(\partial\Omega)}^p &= 
\diam(\Omega)^{d-1-sp}[\nu']_{W^{s,p}(\partial\Omega')}^p
>\diam(\Omega)^{d-1-sp}\mathcal{E}_p(\Omega')\\
&\geq \diam(\Omega)^{d-1-sp}C_{0}\mathcal{P}(\Omega')
=C_{0}\frac{\mathcal{P}(\Omega)}{\diam(\Omega)^{sp}}
\quad\text{for all $s\in(0,1)$.}
\end{split}
\end{equation}
The first inequality above is always strict because 
 \eqref{eq:s0_to_s>0_1} is a strict inequality if $|x'-y'|<1$.

\begin{corollary}\label{coro:all_s}
Let $\Omega\subset\R^d$ be a bounded open set  with Lipschitz boundary. Then,
\begin{equation}\label{frac_s:coro1}
[\nu]_{W^{s,1}(\partial\Omega)}>
|\sph|\frac{\mathcal{P}(\Omega)}{\diam(\Omega)^{s}}
\quad\text{and}\quad
[\nu]_{W^{s,2}(\partial\Omega)}^2>
\frac{4 |\mathbb{S}^{d-2}|}{d-1}\frac{\mathcal{P}(\Omega)}{\diam(\Omega)^{2s}}
\quad\text{for all $s\in(0,1)$.}
\end{equation}

\end{corollary}
\begin{proof}
In the cases $p=1$ and $p=2$, \eqref{conj:ineq_all_p} holds by \eqref{ineq norm1} and \eqref{ineq norm1 main1}, respectively. Then, 
the corollary follows from \eqref{frac_s:almost_coro1}.
\end{proof}

\begin{remark}\label{app:frac_curv}
We may write
$[\nu]_{W^{s,2}(\partial\Omega)}
=\|c_{\paO,2s-1}\|_{L^2(\paO)}$, where
\begin{equation}\label{fsff defi}
c_{\paO,s'}(x):=\Big(\int_{\paO}\frac{|\nu(x)-\nu(y)|^2}{|x-y|^{d+s'}}\,d\sigma(y)\Big)^{1/2}\quad\text{for $\sigma$-a.e. $x\in\paO$.}
\end{equation} 
For $0<s'<1$, $c_{\paO,s'}^2(x)$ is a $s'$-fractional analogue of the squared second fundamental form of
$\paO$ at $x$, that is, the sum of the squares of the principal curvatures. It appears in the second variation of the $s'$-fractional perimeter functional as part of the fractional Jacobi operator; see \cite[(6.11)]{FFMMM} and \cite[(1.6)]{DPW}. As an application of \Cref{coro:all_s}, we get
\begin{equation}
\|c_{\paO,s'}\|^2_{L^2(\paO)}>
\frac{4 |\mathbb{S}^{d-2}|}{d-1}\frac{\mathcal{P}(\Omega)}{\diam(\Omega)^{s'+1}}
\quad\text{for all $s'\in(0,1)$}
\end{equation}
and every bounded open set $\Omega\subset\R^d$ with Lipschitz boundary.
\end{remark}

We now investigate a second simple argument to derive lower bounds for 
$[\nu]_{W^{{s},\mathtt{p}}(\partial\Omega)}$ for certain values of  ${s}\in(0,1)$ and 
$\mathtt{p}>0$ assuming that  \eqref{conj:ineq_all_p} holds for some $p>0$. The idea now is to use H\"older inequality instead of~\eqref{eq:s0_to_s>0_1}. This approach is more interesting than the previous one because it will provide sharp lower bounds for $[\nu]_{W^{1-1/\mathtt{p},\mathtt{p}}(\partial\Omega)}$ for all $\mathtt{p}>1$ when $d=2$ and for 
$[\nu]_{W^{1-2/\mathtt{p},\mathtt{p}}(\partial\Omega)}$ for all $\mathtt{p}>2$ when $d=3$, thanks to \Cref{main thm intro}.

Assume that \eqref{conj:ineq_all_p} holds for some $p>0$. Take any $r>1$. Then, H\"older inequality yields
\begin{equation}\label{frac_s:almost_coro2_2}
\begin{split}
C_{0}\mathcal{P}(\Omega)\leq\mathcal{E}_p(\Omega)
&\leq\Big(\int_{\partial\Omega}\!\int_{\partial\Omega}\frac{|\nu(x)-\nu(y)|^{rp}}{|x-y|^{r(d-1)}} \,d\sigma(x)\,d\sigma(y)\Big)^{1/r}\mathcal{P}(\Omega)^{2-2/r}\\
&=[\nu]_{W^{(1-1/r)(d-1)/p,rp}(\paO)}^p
\mathcal{P}(\Omega)^{2-2/r}.
\end{split}
\end{equation}
Setting 
\begin{equation}\label{frac:eq_def_new_par_1}
{s}:=(1-1/r)(d-1)/p>0\quad\text{and}\quad\mathtt{p}:=rp>0, 
\end{equation}
\eqref{frac_s:almost_coro2_2}  leads to the lower bound
\begin{equation}\label{frac_s:almost_coro2}
\begin{split}
[\nu]_{W^{{s},\mathtt{p}}(\paO)}^{\mathtt{p}}
=[\nu]_{W^{(1-1/r)(d-1)/p,rp}(\paO)}^{rp}
\geq
C_{0}^r\mathcal{P}(\Omega)^{2-r}=
C_{0}^r\mathcal{P}(\Omega)^{1-{s}\mathtt{p}/(d-1)}.
\end{split}
\end{equation}

Next, we analyze the case of equality in \eqref{frac_s:almost_coro2}. For this, we will further assume that $C_0$ denotes the sharp constant in \eqref{conj:ineq_all_p} and that equality in \eqref{conj:ineq_all_p} holds if and only if $\Omega$ is a ball; this is actually the case for $p=1$ and $p=2$ by \Cref{main thm intro}.
On the one hand, if equality in \eqref{frac_s:almost_coro2} holds, in view of \eqref{frac_s:almost_coro2_2} we must have 
$\mathcal{E}_p(\Omega)=C_0\mathcal{P}(\Omega)$, thus 
$\Omega$ must be a ball. Since \eqref{frac_s:almost_coro2} is translation and scale invariant, without loss of generality we can assume that $\Omega=\mathbb{B}^d$, hence $\nu(x)=x$ for all $x\in\partial\Omega=\sph$. On the other hand, we must have equality in H\"older inequality. In our setting, this means that the function
\begin{equation}
(x,y)\mapsto\frac{|\nu(x)-\nu(y)|^p}{|x-y|^{d-1}}=|x-y|^{p+1-d}
\end{equation}
must be constant $(\sigma\times\sigma)$-a.e. in 
$\sph\times\sph$, but this only happens if $p=d-1$. 

The following conclusion summarizes the previous arguments: If \eqref{conj:ineq_all_p} holds for some $p>0$ then \eqref{frac_s:almost_coro2} holds for all $r>1$.  Moreover, if the equality holds in \eqref{conj:ineq_all_p} if and only if $\Omega$ is a ball, then the equality in \eqref{frac_s:almost_coro2} holds for some (and thus for all) $r>1$ if and only if $p=d-1$ and $\Omega$ is a ball.

A final comment is in order. Note that if $p\geq d-1$ then 
$s$ as in \eqref{frac:eq_def_new_par_1} belongs to $(0,1)$ for all $r>1$, hence ${s}$ really represents a degree of fractional smoothness. However, if $p<d-1$ then we get $s\geq1$ for all $r\geq(d-1)/(d-1-p)$. But if $s\geq1$ then \cite[Proposition 2]{Brezis2002} indicates that the left hand side of \eqref{frac_s:almost_coro2} (defined as in \eqref{def:Gagl_s}) is infinite ---since $\nu$ cannot be constant on the connected components of $\paO$ because $\Omega$ is assumed to be bounded---, which means that \eqref{frac_s:almost_coro2} becomes a trivial inequality. To guarantee that we have 
${s}<1$ if $p<d-1$, we must take $r<(d-1)/(d-1-p)$. This is the reason why the following result only addresses the regime $1<r<(d-1)/(d-1-p)$ (and not simply $r>1$) whenever $p<d-1$.

\begin{corollary}\label{coro:sharp_Wsp}
Let $\Omega\subset\R^d$ be a bounded open set with Lipschitz boundary.
\begin{itemize}
\item[$(i)$] If $d=2$, then 
\begin{equation}\label{coro:sharp_Wsp_eq1}
[\nu]_{W^{1-1/\mathtt{p},\mathtt{p}}(\paO)}\geq
2\pi\mathcal{P}(\Omega)^{2/\mathtt{p}-1}
\quad\text{for all $\mathtt{p}>1$}
\end{equation} 
and the equality holds if and only if $\Omega\subset\R^2$ is a disk. In addition,
\begin{equation}
[\nu]_{W^{1/2-1/\mathtt{p},\mathtt{p}}(\paO)}>
2\sqrt{2}\mathcal{P}(\Omega)^{2/\mathtt{p}-1/2}\quad\text{for all $\mathtt{p}>2$.}
\end{equation} 
\item[$(ii)$] If $d=3$, then 
\begin{equation}\label{coro:sharp_Wsp_eq2}
[\nu]_{W^{1-2/\mathtt{p},\mathtt{p}}(\paO)}\geq
2\sqrt{\pi}\mathcal{P}(\Omega)^{2/\mathtt{p}-1/2}
\quad\text{for all $\mathtt{p}>2$}
\end{equation} and the equality holds if and only if $\Omega\subset\R^3$ is a ball. In addition,
\begin{equation}
[\nu]_{W^{2-2/\mathtt{p},\mathtt{p}}(\paO)}>4\pi\mathcal{P}(\Omega)^{2/\mathtt{p}-1}
\quad\text{for all $1<\mathtt{p}<2$.}
\end{equation}
\item[$(iii)$] If $d\geq4$, then 
\begin{equation}
\begin{split}
[\nu]_{W^{(1-1/\mathtt{p})(d-1),\mathtt{p}}(\paO)}
&>|\sph|\mathcal{P}(\Omega)^{2/\mathtt{p}-1}
\quad\text{for all $1<\mathtt{p}<\tfrac{d-1}{d-2}$ and}\\
[\nu]_{W^{(1/2-1/\mathtt{p})(d-1),\mathtt{p}}(\paO)}
&>2\sqrt{\frac{|\mathbb{S}^{d-2}|}{d-1}}
\mathcal{P}(\Omega)^{2/\mathtt{p}-1/2}
\quad\text{for all $2<\mathtt{p}<2\tfrac{d-1}{d-3}$.}
\end{split}
\end{equation}
\end{itemize}
\end{corollary}

\begin{proof}
The proof follows by a straightforward application of the conclusion (and the subsequent comment) stated before the corollary. On the one hand, \Cref{main thm intro} shows that \eqref{conj:ineq_all_p} holds for $p=1$ and $C_0=|\sph|$, and that equality holds in \eqref{conj:ineq_all_p} if and only if $\Omega$ is a ball. Therefore, 
\begin{equation}
\begin{split}
[\nu]_{W^{(1-1/r)(d-1),r}(\paO)}\geq
|\sph|\mathcal{P}(\Omega)^{2/r-1}
\end{split}
\end{equation}
holds for all $1<r<(d-1)/(d-2)$ if $d>2$ and all $r>1$ if 
$d=2$. In addition, equality holds here if and only if $d=2$ and $\Omega$ is a disk. This proves the first inequality in $(i)$ (and the case of equality), the second inequality in $(ii)$, and the first inequality in $(iii)$; recall that now $\mathtt{p}:=rp=r$.

On the other hand, \Cref{main thm intro} shows that \eqref{conj:ineq_all_p} holds for $p=2$ and $C_0=\frac{4 |\mathbb{S}^{d-2}|}{d-1}$, and that equality holds in \eqref{conj:ineq_all_p} if and only if $\Omega$ is a ball. Therefore, 
\begin{equation}
\begin{split}
[\nu]_{W^{(1-1/r)(d-1)/2,2r}(\paO)}^{2}\geq
\frac{4 |\mathbb{S}^{d-2}|}{d-1}
\mathcal{P}(\Omega)^{2/r-1}
\end{split}
\end{equation}
holds for all $1<r<(d-1)/(d-3)$ if $d>3$ and all $r>1$ if 
$d\leq 3$. In addition, equality holds here if and only if $d=3$ and $\Omega$ is a ball.
This proves the second inequality in $(i)$, the first inequality in $(ii)$ (and the case of equality), and the second inequality in 
$(iii)$; recall that now $\mathtt{p}:=rp=2r$.
\end{proof}

\begin{remark}\label{op:Ep}
Given   
${s}\in(0,1)$ and $\mathtt{p}>0$, take  
\begin{equation}\label{frac:eq_def_new_par}
r:=1+{s}\mathtt{p}/(d-1)>1\quad\text{and}\quad
p:=\mathtt{p}(d-1)/({s}\mathtt{p}+d-1)>0.
\end{equation}
This relation between the pairs $(r,p)$ and 
$({s},\mathtt{p})$ is equivalent to the one in \eqref{frac:eq_def_new_par_1}, which was used to write \eqref{frac_s:almost_coro2}. Therefore, for $p$ as in \eqref{frac:eq_def_new_par}, if \eqref{conj:ineq_all_p} holds then \eqref{frac_s:almost_coro2} holds for $r$ as in \eqref{frac:eq_def_new_par}. 
Rewriting this conclusion in terms of $s$ and $\mathtt{p}$ we deduce that,
given ${s}\in(0,1)$ and $\mathtt{p}>0$, if
$\mathcal{E}_{\mathtt{p}(d-1)/({s}\mathtt{p}+d-1)}(\Omega)\geq C\mathcal{P}(\Omega)$ 
for some $C>0$ then
$[\nu]_{W^{{s},\mathtt{p}}(\paO)}^{\mathtt{p}}\geq
C^{1+{s}\mathtt{p}/(d-1)}
\mathcal{P}(\Omega)^{1-{s}\mathtt{p}/(d-1)}.$
\end{remark}

We finish this section posing some open problems which arise from the previous discussions. The quest for sharp lower bounds for $[\nu]_{W^{{s},\mathtt{p}}(\paO)}$ in terms of $\mathcal{P}(\Omega)$ for all ${s}\in(0,1)$ and $\mathtt{p}>0$ is a natural one, in particular in view of the connection to the fractional squared second fundamental form described in \Cref{app:frac_curv}. It would also be interesting to find sharp lower bounds for $\mathcal{E}_p(\Omega)$ for all $p>0$. Thanks to \Cref{op:Ep}, they would provide lower bounds for 
$[\nu]_{W^{{s},\mathtt{p}}(\paO)}$ for all ${s}\in(0,1)$ and $\mathtt{p}>0$ ---although the latter might not be attained by any domain $\Omega$, as is the case in \Cref{coro:sharp_Wsp} $(iii)$. Finally, since all these lower bounds only involve $\mathcal{P}(\Omega)$, it would be interesting to extend them to sets of finite perimeter.

\section{Proof of \Cref{main thm intro}}\label{s monotonicity perimeter}

Before entering into the details, let us describe the strategy of the proof for the benefit of the reader. First of all, we remark that our proof of \eqref{ineq norm1} is much simpler than the one of~\eqref{ineq norm1 main1}. Essentially, \eqref{ineq norm1} will follow by a well-known identity for the double layer potential of the Laplace operator (see \cite[Proposition 3.19]{Folland} for $C^2$~domains) and Cauchy-Schwarz inequality. The proof will be carried out at the end of \Cref{ss monotonicity perimeter}.

Regarding \eqref{ineq norm1 main1}, it will be a straightforward application of the following result, a monotonicity formula involving Bessel potentials. In more detail, given a bounded open set $\Omega\subset\R^d$ with Lipschitz boundary such that $\mathcal{E}_2(\Omega)<+\infty$ and $a>0$, we define
\begin{equation}\label{defi Phi intro}
\begin{split}
\Phi(\Omega,a)&:=\frac{1}{(2\pi)^{d/2}}
\int_{\paO}\!\int_{\paO}\Big(\int_{a|x-y|}^{+\infty}
t^{d/2-1}K_{d/2-1}(t)\,dt\Big)\frac{|\nu(x)-\nu(y)|^2}{|x-y|^{d-1}}
\,d\sigma(x)\,d\sigma(y)\\
&\quad+\frac{4a^{d/2}}{(2\pi)^{d/2}}\int_{\Omega^c}\!\int_{\Omega}
\Big\{\Big(1-\frac{d}{2}\Big)
\frac{K_{d/2-1}(a|x-y|)}{|x-y|^{1+d/2}}
-a\,\frac{K_{d/2-1}'(a|x-y|)}{|x-y|^{d/2}}\Big\}\,dx\,dy.
\end{split}
\end{equation}
Here, $K_{d/2-1}$ denotes the modified Bessel function of the second kind and order $d/2-1$; see \Cref{ss Bessel ineq} for more details.
We will show that $\Phi(\Omega,a)$ is monotone in $a\in(0,+\infty)$.
To prove this monotonicity, we will first find a sharp inequality between a solid integral and a boundary integral involving  $K_{d/2-1}$; see \Cref{thm fund ineq} below.
The proof of this inequality is mainly based on the Gauss-Green theorem and the Reflection Lemma, which characterizes the balls of 
$\R^d$; this part is developed in \Cref{ss Bessel ineq}. 
Using this sharp inequality, we will prove that, as a function of $a\in(0,+\infty)$, if $\Omega$ is a ball then
$\Phi(\Omega,a)$ is constant and if $\Omega$ is not a ball then 
$\Phi(\Omega,a)$ is strictly decreasing. Moreover, we will compute its limit when $a\downarrow0$ and $a\uparrow+\infty$. In the first case we will get 
$\mathcal{E}_2(\Omega)$ and in the second one we will obtain $\mathcal{P}(\Omega)$, modulo some precise constants depending only on $d$.

The following proposition, which summarizes these conclusions, is the main technical result from which \eqref{ineq norm1 main1} (and also the case of equality) follows as a corollary. Its proof and the one of \eqref{ineq norm1 main1} are given in \Cref{ss monotonicity perimeter}.

\begin{proposition}\label{main thm}
Let $\Omega\subset\R^d$ be a bounded open set with Lipschitz boundary such that $\mathcal{E}_2(\Omega)<+\infty$. Then, 
\begin{eqnarray}
&\frac{\partial}{\partial a}\Phi(\Omega,a)
\leq 0\quad\text{for all }a>0, \label{monotony}\\
&\lim_{a\downarrow0}\Phi(\Omega,a)
=\kappa\,\mathcal{E}_2(\Omega),
\quad\text{and}\quad 
\lim_{a\uparrow+\infty}\Phi(\Omega,a)
=\kappa \frac{4 |\mathbb{S}^{d-2}|}{d-1}\mathcal{P}(\Omega) \label{limit zero_infty},
\end{eqnarray}
where $\kappa:=(2\pi)^{-d/2}\int_{0}^{+\infty}t^{d/2-1}K_{d/2-1}(t)\,dt$ is a positive and finite constant. Moreover, the equality in \eqref{monotony} holds for some (and thus for all) $a>0$ if and only if $\Omega$ is a ball.  
\end{proposition} 

\begin{remark}
The assumption $\mathcal{E}_2(\Omega)<+\infty$ in \Cref{main thm} is meaningful, since there exists a bounded open set $\Omega$ with Lipschitz boundary such that $\mathcal{E}_2(\Omega)=+\infty$; see \Cref{s:examples}.
\end{remark}

A final comment is in order. The reader familiar with heat-flow monotonicity techniques will observe similarities with our approach. Several integral inequalities in euclidean analysis can be proved using adequate (sub/super)solutions of the heat equation $\partial_t-\Delta$, for which certain monotone behavior holds in $t>0$. Then, the evaluation at different times yields an inequality which, in many situations, generates sharp constants and identifies optimizers; see \cite{B} for a survey on the subject.
In certain cases the heat operator is replaced by other differential operators. Roughly speaking, in this section we will use the operator 
$-\Delta+a^2$ to construct a flow as $a$ ranges in 
$(0,+\infty)$.

\subsection{The key sharp inequality involving Bessel potentials}\label{ss Bessel ineq}

We begin this section by introducing a suitable fundamental solution of $-\Delta+a^2$ for $a>0$ that will be used in the sequel.
Given $\alpha\geq0$, let $K_{\alpha}$ denote the modified Bessel function of the second kind and order $\alpha$; see \cite{AbraStegun,Watson} for the definition and properties. The function 
$K_{\alpha}$ satisfies the differential equation
\begin{equation}\label{diff eq Bessel}
t^2K_{\alpha}''(t)=(t^2+\alpha^2)K_{\alpha}(t)-tK_{\alpha}'(t)
\quad\text{for all }t>0.
\end{equation}
Set
\begin{equation}\label{def G}
G(x):=\frac{|x|^{1-d/2}}{(2\pi)^{d/2}}\,K_{d/2-1}(|x|),\qquad
H(x):=\frac{|x|^{1-d/2}}{(2\pi)^{d/2}}\,K_{d/2-1}'(|x|)
\end{equation}
for $x\in\R^d\setminus\{0\}$ and
\begin{equation}\label{def Ga}
G_a(x):=a^{d-2}G(ax),\qquad
H_a(x):=a^{d-2}H(ax)
\end{equation}
for $a>0$ and $x\in\R^d\setminus\{0\}$. From \cite[9.6.24]{AbraStegun} or \cite[(5) in page 181]{Watson} 
we know that 
\begin{equation}\label{K int exp form}
K_{\alpha}(t)=\int_0^{+\infty} e^{-t\cosh r}\cosh(\alpha r)\,dr>0
\quad\text{for all $t>0$},
\end{equation}
thus  $G_a(x)>0$ for all $x\in\R^d\setminus\{0\}$. Moreover,
since $G(x)$ behaves as $|x|^{2-d}$ for $d\geq3$ and as $-\ln |x|$ for $d=2$ when 
$|x|\to 0$, and decays exponentially fast as $|x|\to+\infty$, we see that $G_a\in L^1(\R^d)$. It is also well known that, given $f\in L^p(\R^d)$ ($1\leq p \leq+\infty$), the function
\begin{equation}
(G_a*f)(x):=\int_{\R^d}G_a(x-y)f(y)\,dy
\end{equation}
belongs to $L^p(\R^d)$ and satisfies $(-\Delta+a^2)(G_a*f)=f$.  Therefore, $(-\Delta+a^2)G_a=\delta_0$ in the sense of distributions, where $\delta_0$ denotes the Dirac measure centered at the origin.  In particular, taking $f=1$ it is clear that
\begin{equation}\label{Ga int prop}
a^2\int_{\R^d}G_a(x)\,dx=(-\Delta+a^2)(G_a*f)=1
\end{equation}
---a description of all these properties of $G_a$ can be found in 
\cite[Section 7.4]{Teschl}, for example.
In the sequel, we will refer to $G_a$ as the Bessel potential. The next lemma contains some formulas involving $G_a$ to be used later on.\footnote{For $d=3$, we have $K_{d/2-1}(t)=K_{1/2}(t)=\sqrt{\frac{\pi}{2t}}e^{-t}$ for all $t>0$ and $G_a(x)=\frac{e^{-a|x|}}{4\pi|x|}$ for all $x\in\R^3\setminus\{0\}$. These simple expressions of the Bessel function and the fundamental solution of $-\Delta+a^2$ may help the reader to follow the computations developed in \Cref{s monotonicity perimeter} for the general case $d\geq2$, since all of them become completely explicit for $d=3$.}
\begin{lemma}\label{formula daGa}
The following identities hold for all $a>0$ and $x\in\R^d\setminus\{0\}$:
\begin{eqnarray}
&&\frac{\partial}{\partial a}\Big(\frac{G_a(x)}{|x|^2}\Big)
=\frac{H_a(x)}{|x|}+\Big(\frac{d}{2}-1\Big)\frac{G_a(x)}{a|x|^2},
\label{formula daGa eq1}\\
&&\frac{\partial}{\partial a}\Big(a^{3-d}\frac{\partial}{\partial a}
\Big(\frac{G_a(x)}{|x|^2}\Big)\Big)=a^{3-d}G_a(x).\label{formula daGa eq2}
\end{eqnarray}
\end{lemma}

\begin{proof}
Using \eqref{def G} and \eqref{def Ga} we compute
\begin{equation}
\begin{split}
\frac{\partial}{\partial a}\Big(a\,\frac{G_a(x)}{|x|^2}\Big)
&=\frac{\partial}{\partial a}\Big(a^{d/2}\frac{|x|^{-1-d/2}}{(2\pi)^{d/2}}\,K_{d/2-1}(a|x|)\Big)\\
&=a^{d/2}\frac{|x|^{-d/2}}{(2\pi)^{d/2}}\,K_{d/2-1}'(a|x|)
+\frac{d}{2}\,a^{d/2-1}\frac{|x|^{-1-d/2}}{(2\pi)^{d/2}}\,K_{d/2-1}(a|x|)\\
&=a\,\frac{H_a(x)}{|x|}+d\,\frac{G_a(x)}{2|x|^2}.
\end{split}
\end{equation}
From here, \eqref{formula daGa eq1} follows directly. 
Then, using \eqref{formula daGa eq1} and \eqref{diff eq Bessel}, we see that
\begin{equation}
\begin{split}
\frac{\partial^2}{\partial a^2}\Big(\frac{G_a(x)}{|x|^2}\Big)
&=\frac{\partial}{\partial a}\Big(\frac{H_a(x)}{|x|}+\Big(\frac{d}{2}-1\Big)\frac{G_a(x)}{a|x|^2}\Big)\\
&=\frac{\partial}{\partial a}\Big(a^{d/2-1}\frac{|x|^{-d/2}}{(2\pi)^{d/2}}\,K_{d/2-1}'(a|x|)
+\Big(\frac{d}{2}-1\Big)a^{d/2-2}\frac{|x|^{-1-d/2}}{(2\pi)^{d/2}}\,K_{d/2-1}(a|x|)\Big)\\
&=(d-3)a^{d/2-2}\frac{|x|^{-d/2}}{(2\pi)^{d/2}}\,K_{d/2-1}'(a|x|)
+a^{d/2-1}\frac{|x|^{1-d/2}}{(2\pi)^{d/2}}K_{d/2-1}(a|x|)\\
&\quad+\frac{1}{2}\,(d-2)(d-3)a^{d/2-3}\frac{|x|^{-1-d/2}}{(2\pi)^{d/2}}K_{d/2-1}(a|x|)\\
&=\frac{d-3}{a^2}\Big(a\,\frac{H_a(x)}{|x|}+(d-2)\frac{G_a(x)}{2|x|^2}\Big)+G_a(x)
=\frac{d-3}{a}\frac{\partial}{\partial a}\Big(\frac{G_a(x)}{|x|^2}\Big)+G_a(x).
\end{split}
\end{equation}
Therefore,
\begin{equation}
\Big(\frac{\partial}{\partial a}-\frac{d-3}{a}\Big)\frac{\partial}{\partial a}
\Big(\frac{G_a(x)}{|x|^2}\Big)=G_a(x)
\end{equation}
and thus
\begin{equation}
\frac{\partial}{\partial a}\Big(a^{3-d}\frac{\partial}{\partial a}
\Big(\frac{G_a(x)}{|x|^2}\Big)\Big)
=a^{3-d}\Big(\frac{\partial}{\partial a}+\frac{3-d}{a}\Big)\frac{\partial}{\partial a}
\Big(\frac{G_a(x)}{|x|^2}\Big)=a^{3-d}G_a(x),
\end{equation}
which yields \eqref{formula daGa eq2}.
\end{proof}

The following nonlocal perimeter defined in terms of the Bessel potential is a key ingredient in the proof of \Cref{main thm}. Given a bounded open set $\Omega\subset\R^d$ with Lipschitz boundary and $a>0$, set
\begin{equation}\label{Lambda definition}
\Lambda(\Omega,a):=\int_{\Omega^c}\!\int_{\Omega}G_a(x-y)\,dx\,dy.
\end{equation} 
Using that  $G_a$ is positive and \eqref{Ga int prop}, we can trivially estimate 
$0<a^2\Lambda(\Omega,a)< |\Omega|$, where $|\Omega|$ denotes the Lebesgue measure of $\Omega$. Indeed, thanks to \eqref{Ga int prop}, we have 
\begin{equation}
\begin{split}
0<a^2\Lambda(\Omega,a)&=
a^2\int_{\R^d}\!\int_{\Omega}G_a(x-y)\,dx\,dy
-a^2\int_{\Omega}\!\int_{\Omega}G_a(x-y)\,dx\,dy\\
&=|\Omega|
-a^2\int_{\Omega}\!\int_{\Omega}G_a(x-y)\,dx\,dy<|\Omega|.
\end{split}
\end{equation}

Combining the identity $(-\Delta+a^2)G_a=\delta_0$ with the Gauss-Green theorem, in the following lemma we provide two identities which relate $\Lambda(\Omega,a)$ to certain  boundary integrals. They will be used to prove  \Cref{thm fund ineq}.

\begin{lemma}
Let $\Omega\subset\R^d$ be a bounded open set with Lipschitz boundary. Then, the following holds for all $a>0$:
\begin{eqnarray}
&&a^2\Lambda(\Omega,a)=\int_{\paO}\!\int_{\paO}G_a(x-y)\,
\nu(x)\cdot\nu(y)\,d\sigma(x)\,d\sigma(y)\label{formula Ga ident1},\\
&&\begin{split}
\int_{\paO}\!\int_{\paO}G_a(x-y)
\Big(\nu(x)&\cdot\frac{x-y}{|x-y|}\Big)\Big(\nu(y)\cdot\frac{x-y}{|x-y|}\Big)
\,d\sigma(x)\,d\sigma(y)\label{formula Ga ident2}\\
&=a^2\Lambda(\Omega,a)
+a(d-1)\int_{\Omega^c}\!\int_{\Omega}
\frac{\partial}{\partial a}\Big(\frac{G_a(x-y)}{|x-y|^2}\Big)\,dx\,dy.
\end{split}
\end{eqnarray}
\end{lemma}

\begin{proof}
Recall that $(-\Delta+a^2)G_a=\delta_0$. 
Therefore, for $x,\,y\in\R^d$ with $x\neq y$ we have
\begin{equation}\label{formula Ga ident1 eq1}
a^2G_a(x-y)=(\Delta G_a)(x-y)=\ddiv_x\big((\nabla G_a)(x-y)\big)
=-\ddiv_x\nabla_y \big(G_a(x-y)\big),
\end{equation}
where $\ddiv_x$ and $\nabla_y$ mean the divergence and the gradient on the $x$ and $y$ variables, respectively. From \eqref{formula Ga ident1 eq1} and the Gauss-Green theorem applied twice we get \eqref{formula Ga ident1}.

We now focus on \eqref{formula Ga ident2}. A computation shows that
\begin{equation}\label{formula Ga ident2 eq1}
\begin{split}
\ddiv_x\Big\{G_a(x&-y)\Big(\nu(y)\cdot\frac{x-y}{|x-y|^2}\Big)(x-y)\Big\}\\
&=\Big((x-y)\cdot(\nabla G_a)(x-y)+(d-1)G_a(x-y)\Big)
\Big(\nu(y)\cdot\frac{x-y}{|x-y|^2}\Big),
\end{split}
\end{equation}
and that
\begin{equation}\label{formula Ga ident2 eq2}
\begin{split}
\ddiv_y\Big\{&\Big((x-y)\cdot(\nabla G_a)(x-y)+(d-1)G_a(x-y)\Big)
\frac{x-y}{|x-y|^2}\Big\}\\
&=-\frac{1}{|x-y|^2}\Big\{(d-1)(d-2)G_a(x-y)+2(d-1)(x-y)\cdot(\nabla G_a)(x-y)\\
&\quad+(x-y)[D^2G_a(x-y)](x-y)^t\Big\};
\end{split}
\end{equation}
by $(x-y)[D^2G_a(x-y)](x-y)^t$ we mean
$\sum_{1\leq i,\,j\leq d} (x_i-y_i)(x_j-y_j)(\partial_i\partial_j G_a)(x-y)$.
Therefore, by the Gauss-Green theorem, \eqref{formula Ga ident2 eq1}, and \eqref{formula Ga ident2 eq2}, we get
\begin{equation}\label{formula Ga ident2 eq3}
\begin{split}
\int_{\paO}&\!\int_{\paO}
G_a(x-y)\Big(\nu(x)\cdot\frac{x-y}{|x-y|}\Big)\Big(\nu(y)\cdot\frac{x-y}{|x-y|}\Big)
\,d\sigma(x)\,d\sigma(y)\\
&=\int_{\Omega^c}\!\int_{\Omega}
\frac{1}{|x-y|^2}\Big\{(d-1)(d-2)G_a(x-y)+2(d-1)(x-y)\cdot(\nabla G_a)(x-y)\\
&\quad+(x-y)[D^2G_a(x-y)](x-y)^t\Big\}\,dx\,dy.
\end{split}
\end{equation}

We can compute $\nabla G_a$ and $D^2G_a$ on the right hand side of \eqref{formula Ga ident2 eq3} using the definition of $G_a$ in terms of the Bessel function $K_{d/2-1}$. More precisely, by \eqref{def Ga} and \eqref{def G},
\begin{equation}
\begin{split}
\nabla G_a(x)=a^{d-1}(\nabla G)(ax)=
\frac{a^{d/2-1}}{(2\pi)^{d/2}}\Big(\Big(1-\frac{d}{2}\Big){K_{d/2-1}(a|x|)}
+a|x|K_{d/2-1}'(a|x|)\Big)\frac{x}{|x|^{1+d/2}}
\end{split}
\end{equation}
and
\begin{equation}
\begin{split}
\partial_i\partial_j&G_a(x)=\frac{a^{d/2-1}}{(2\pi)^{d/2}}\Big\{
\frac{\delta_{i,j}}{|x|^{1+d/2}}\Big(\Big(1-\frac{d}{2}\Big){K_{d/2-1}(a|x|)}
+a|x|K_{d/2-1}'(a|x|)\Big)\\
&+\frac{x_ix_j}{|x|^{3+d/2}}\Big(\Big(\frac{d^2}{4}-1\Big){K_{d/2-1}(a|x|)}+
(1-d)a|x|{K_{d/2-1}'(a|x|)}+a^2|x|^2K_{d/2-1}''(a|x|)\Big)\Big\},
\end{split}
\end{equation}
where $\delta_{i,j}=1$ if $i=j$ and $\delta_{i,j}=0$ if $i\neq j$. 
With this in hand, we obtain
\begin{equation}\label{formula Ga ident2 eq4}
\begin{split}
(d-1)(d-2)G_a(x-y)&+2(d-1)(x-y)\!\cdot\!(\nabla G_a)(x-y)
+(x-y)[D^2G_a(x-y)](x-y)^t\\
&=\frac{a^{d/2-1}}{(2\pi)^{d/2}}|x-y|^{1-d/2}
\Big\{\frac{1}{4}\,d(d-2)K_{d/2-1}(a|x-y|)\\
&\quad+da|x-y|K_{d/2-1}'(a|x-y|)
+a^2|x-y|^2K_{d/2-1}''(a|x-y|)\Big\}.
\end{split}
\end{equation}
Using \eqref{diff eq Bessel} we see that \eqref{formula Ga ident2 eq4} can be rewritten as
\begin{equation}\label{formula Ga ident2 eq5}
\begin{split}
(d-1)(d-2)G_a(x-y)&+2(d-1)(x-y)\!\cdot\!(\nabla G_a)(x-y)
+(x-y)[D^2G_a(x-y)](x-y)^t\\
&=\frac{a^{d/2-1}}{(2\pi)^{d/2}}|x-y|^{1-d/2}\Big\{(d-1)a|x-y|K_{d/2-1}'(a|x-y|)\\
&\quad+\Big(a^2|x-y|^2+\frac{1}{2}\,(d-1)(d-2)\Big)K_{d/2-1}(a|x-y|)\Big\}.
\end{split}
\end{equation}
From \eqref{formula Ga ident2 eq5}, \eqref{def G}, and \eqref{def Ga} we deduce that
\begin{equation}
\begin{split}
(d-1)(d-2)G_a&(x-y)+2(d-1)(x-y)\cdot(\nabla G_a)(x-y)
+(x-y)[D^2G_a(x-y)](x-y)^t\\
&=(d-1)a|x-y|H_a(x-y)
+\Big(a^2|x-y|^2+\frac{1}{2}\,(d-1)(d-2)\Big)G_a(x-y).
\end{split}
\end{equation}
Plugging this into \eqref{formula Ga ident2 eq3} we conclude that
\begin{equation}
\begin{split}
\int_{\paO}\!\int_{\paO}
G_a&(x-y)\Big(\nu(x)\cdot\frac{x-y}{|x-y|}\Big)\Big(\nu(y)\cdot\frac{x-y}{|x-y|}\Big)
\,d\sigma(x)\,d\sigma(y)\\
&=\int_{\Omega^c}\!\int_{\Omega}\Big\{(d-1)a\,\frac{H_a(x-y)}{|x-y|}
+\Big(a^2+\frac{(d-1)(d-2)}{2|x-y|^2}\Big)G_a(x-y)\Big\}\,dx\,dy\\
&=a^2\Lambda(\Omega,a)
+\int_{\Omega^c}\!\int_{\Omega}(d-1)\Big\{a\,\frac{H_a(x-y)}{|x-y|}
+\Big(\frac{d}{2}-1\Big)\frac{G_a(x-y)}{|x-y|^2}\Big\}\,dx\,dy,
\end{split}
\end{equation}
which gives \eqref{formula Ga ident2} thanks to \eqref{formula daGa eq1}.
\end{proof}

The following is the main result in this section and provides a sharp inequality, only attained when $\Omega$ is a ball, which relates a solid and a boundary integral given in terms of the Bessel potential. From this sharp inequality we will extract the monotone behavior of the function $a\mapsto\Phi(\Omega,a)$ described in \Cref{main thm}  which will lead to the proof of \eqref{ineq norm1 main1}.
\begin{lemma}\label{thm fund ineq}
Let $\Omega\subset\R^d$ be a bounded open set with Lipschitz boundary. Then,
\begin{equation}
\begin{split}
0\leq4\int_{\Omega^c}\!\int_{\Omega}\frac{\partial}{\partial a}
\Big(a^{2}\frac{\partial}{\partial a}&
\Big(\frac{G_a(x-y)}{|x-y|^2}\Big)\Big)\,dx\,dy
+ \int_{\paO}\!\int_{\paO}G_a(x-y)|\nu(x)-\nu(y)|^2
\,d\sigma(x)\,d\sigma(y)
\end{split}
\end{equation}
for all $a>0$. The equality holds for some (and thus for all) $a>0$ if and only if $\Omega$ is a ball.
\end{lemma}

\begin{proof}
Using \eqref{formula Ga ident1}, we can split
\begin{equation}\label{CS eq3}
\begin{split}
a^2\Lambda(\Omega,a)
&=\int_{\paO}\!\int_{\paO}G_a(x-y)\,\nu(x)\cdot\nu(y)\,d\sigma(x)\,d\sigma(y)\\
&=\int_{\paO}\!\int_{\paO}G_a(x-y)\,\nu(x)\cdot
\Big(\nu(y)-2\,\frac{\nu(y)\cdot(x-y)}{|x-y|^2}\,(x-y)\Big)\,d\sigma(x)\,d\sigma(y)\\
&\quad+2\int_{\paO}\!\int_{\paO}G_a(x-y)\Big(\nu(x)\cdot\frac{x-y}{|x-y|}\Big)
\Big(\nu(y)\cdot\frac{x-y}{|x-y|}\Big)\,d\sigma(x)\,d\sigma(y)
\\
&=:I_1+I_2.
\end{split}
\end{equation}

Form \eqref{formula Ga ident2}, we see that
\begin{equation}\label{CS eq3_I2}
\begin{split}
I_2=2a^2\Lambda(\Omega,a)
+2a(d-1)\int_{\Omega^c}\!\int_{\Omega}
\frac{\partial}{\partial a}\Big(\frac{G_a(x-y)}{|x-y|^2}\Big)\,dx\,dy.
\end{split}
\end{equation}

Regarding $I_1$, note that 
\begin{equation}
\begin{split}
\Big|\nu(y)-2\,\frac{\nu(y)\cdot(x-y)}{|x-y|^2}\,(x-y)\Big|^2
=1,
\end{split}
\end{equation}
hence the Cauchy-Schwarz inequality applied to the integrand in $I_1$ (using that $G_a$ is positive) shows that
\begin{equation}\label{CS eq1}
\begin{split}
I_1&\leq\int_{\paO}\!\int_{\paO}G_a(x-y)\,
|\nu(x)|\,\Big|\nu(y)-2\,\frac{\nu(y)\cdot(x-y)}{|x-y|^2}\,(x-y)\Big|
\,d\sigma(x)\,d\sigma(y)\\
&=\int_{\paO}\!\int_{\paO}G_a(x-y)\,d\sigma(x)\,d\sigma(y).
\end{split}
\end{equation}

Let us analyze the case of equality in this last estimate. On the one hand, it is a simple exercise to show that if $\Omega$ is a ball then equality in \eqref{CS eq1} holds, since in this case
\begin{equation}
\nu(x)=\nu(y)-2\,\frac{\nu(y)\cdot(x-y)}{|x-y|^2}\,(x-y)\quad\text{for all }x,y\in\paO\text{ with $x\neq y$}.
\end{equation}
On the other hand, if $\Omega$ is a bounded open set with Lipschitz boundary for which equality in \eqref{CS eq1} holds, then
there exists $N\subset\paO$ with $\sigma(N)=0$ such that for all $y\in \paO\setminus N$ it holds that
\begin{equation}\label{CS eq2}
\nu(x)=\nu(y)-2\,\frac{\nu(y)\cdot(x-y)}{|x-y|^2}\,(x-y)\quad\text{for $\sigma$-a.e. }x\in\paO.
\end{equation}
Fixing $y\in \paO\setminus N$, \eqref{CS eq2} shows that, after eventually modifying $\nu$ on a set of $\sigma$-measure zero, we have $\nu\in C^0(\paO,\R^d)$. Now, from \cite[Theorem 2.19]{HMT2007} we deduce that
 $\Omega$ has $C^1$~boundary, and then, combining \eqref{CS eq2} with \cite[Lemma 5.3]{Brock2007} ---or from the arguments in the proof of \cite[Theorem 4.17]{HMMPT}---, we conclude that $\Omega$ is a ball (since we are assuming that $\Omega$ is bounded). 
 
Therefore, from \eqref{CS eq3}, \eqref{CS eq3_I2}, and \eqref{CS eq1} (taking into account also the case of equality), we see that
\begin{equation}\label{CS eq4}
\begin{split}
0\leq \int_{\paO}\!\int_{\paO}G_a(x-y)\,d\sigma(x)\,d\sigma(y)
+a^2\Lambda(\Omega,a)
+2a(d-1)\int_{\Omega^c}\!\int_{\Omega}
\frac{\partial}{\partial a}\Big(\frac{G_a(x-y)}{|x-y|^2}\Big)\,dx\,dy,
\end{split}
\end{equation}
and the equality holds if and only if $\Omega$ is a ball. Thanks to \eqref{formula Ga ident1}, we can rewrite \eqref{CS eq4} as
\begin{equation}\label{CS eq5}
\begin{split}
-2a(d-1)\!\int_{\Omega^c}\!\int_{\Omega}
\frac{\partial}{\partial a}\Big(\frac{G_a(x-y)}{|x-y|^2}\Big)\,dx\,dy
\leq \int_{\paO}\!\int_{\paO}\!G_a(x-y)\big(1+\nu(x)\cdot\nu(y)\big)\,d\sigma(x)\,d\sigma(y).
\end{split}
\end{equation}
Subtracting $2a^2\Lambda(\Omega,a)$ on both sides of \eqref{CS eq5} and using \eqref{formula daGa eq2} and \eqref{formula Ga ident1}, we get
\begin{equation}
\begin{split}
-2\int_{\Omega^c}\!\int_{\Omega}\frac{\partial}{\partial a}
\Big(a^{2}\frac{\partial}{\partial a}
\Big(\frac{G_a(x-y)}{|x-y|^2}&\Big)\Big)\,dx\,dy
=-2\int_{\Omega^c}\!\int_{\Omega}\frac{\partial}{\partial a}
\Big(a^{d-1}a^{3-d}\frac{\partial}{\partial a}
\Big(\frac{G_a(x-y)}{|x-y|^2}\Big)\Big)dx\,dy\\
&=-2(d-1)a^{d-2}\int_{\Omega^c}\!\int_{\Omega}a^{3-d}
\frac{\partial}{\partial a}\Big(\frac{G_a(x-y)}{|x-y|^2}\Big)\,dx\,dy\\
&\quad-2a^{d-1}\int_{\Omega^c}\!\int_{\Omega}\frac{\partial}{\partial a}\Big(a^{3-d}\frac{\partial}{\partial a}
\Big(\frac{G_a(x-y)}{|x-y|^2}\Big)\Big)\,dx\,dy\\
&=-2(d-1)a\int_{\Omega^c}\!\int_{\Omega}
\frac{\partial}{\partial a}\Big(\frac{G_a(x-y)}{|x-y|^2}\Big)\,dx\,dy
-2a^{2}\Lambda(\Omega,a)\\
&\leq \int_{\paO}\!\int_{\paO}G_a(x-y)\big(1-\nu(x)\cdot\nu(y)\big)
\,d\sigma(x)\,d\sigma(y)\\
&=\frac{1}{2}\int_{\paO}\!\int_{\paO}G_a(x-y)|\nu(x)-\nu(y)|^2
\,d\sigma(x)\,d\sigma(y),
\end{split}
\end{equation}
which proves the inequality in the statement of the lemma.
As before, the equality holds if and only if $\Omega$ is a ball. 
\end{proof}

\subsection{Proof of \Cref{main thm} and 
\Cref{main thm intro}}\label{ss monotonicity perimeter}

In this section we will deal with the function $\Phi(\Omega,a)$ introduced in \eqref{defi Phi intro}. We will prove that it is monotone in $a$ thanks to \Cref{thm fund ineq}. Actually, we will show that $\Phi(\Omega,a)$ is constant if and only if $\Omega$ is a ball, and it is strictly decreasing otherwise. We will then compute its limit as 
$a\downarrow0$ and $a\uparrow+\infty$ obtaining $\mathcal{E}_2(\Omega)$ and $\mathcal{P}(\Omega)$, respectively, modulo some precise constants. These are the contents of \Cref{main thm}, from which \eqref{ineq norm1 main1} will follow easily, as well as the case of equality. At the end of this section, we will also give the proof of \eqref{ineq norm1} and the case of equality, completing the proof of \Cref{main thm intro}.

In order to study the asymptotic behavior of $\Phi(\Omega,a)$ with respect to $a$, we introduce two auxiliary functions related to the Bessel potential.
Set
\begin{equation}\label{def F}
W(x):=\frac{1}{(2\pi)^{d/2}}\int_{|x|}^{+\infty}t^{d/2-1}K_{d/2-1}(t)\,dt,
\qquad
F(x):=\Big(1-\frac{d}{2}\Big)\frac{G(x)}{|x|}-H(x)
\end{equation}
for $x\in\R^d\setminus\{0\}$ and
\begin{equation}\label{def Fa}
W_a(x):=W(ax),\qquad
F_a(x):=a^{d}F(ax)
\end{equation}
for $a>0$ and $x\in\R^d\setminus\{0\}$,
where $G$ and $H$ are given in \eqref{def G}.
The following lemma shows the relation between $W_a$, $F_a$, and $G_a$, as well as some properties that will be useful for computing the above-mentioned limits with respect to $a$.

\begin{lemma}\label{lemma W F}
The following identities hold for all $a>0$ and $x\in\R^d\setminus\{0\}$:
\begin{eqnarray}
\int_a^{+\infty}G_s(x)\,ds=\frac{W_a(x)}{|x|^{d-1}},\label{int Gs is W}\\
-a^2\frac{\partial}{\partial a}\Big(\frac{G_a(x)}{|x|^2}\Big)
=\frac{F_a(x)}{|x|}.\label{deriv F is double deriv G}
\end{eqnarray}
Furthermore, the following holds:
\begin{itemize}
\item[$(i)$]
$\lim_{a\uparrow+\infty}W_a(x)=0$ and 
$\lim_{a\downarrow 0}W_a(x)=\kappa$ for all $x\in\R^d\setminus\{0\}$, where 
\begin{equation}\label{W(0)}
\kappa:=\frac{1}{(2\pi)^{d/2}}\int_{0}^{+\infty}t^{d/2-1}K_{d/2-1}(t)\,dt
=\frac{2^{1-d/2}}{|\sph|}\int_{0}^{+\infty}
\frac{\cosh\big((d/2-1) t\big)}{(\cosh t)^{d/2}}\,dt
\end{equation}
is a positive and finite constant.
\item[$(ii)$]
$F(x)>0$ for all $x\in\R^d\setminus\{0\}$. Moreover, there exists a constant $C_d>0$ depending only on $d$ such that
\begin{equation}
F(x)\leq C_d\Big(\frac{1}{|x|^{d-1}}\chi_{(0,1)}(|x|)+\frac{e^{-|x|}}{|x|^{(d-1)/2}}\chi_{[1+\infty)}(|x|)\Big)\quad\text{for all $x\in\R^d\setminus\{0\}$}.
\end{equation}
In particular, $\int_{\R^d}F(x)\,dx<+\infty$.
\end{itemize}
\end{lemma}

\begin{proof}
Using \eqref{def Ga}, \eqref{def G}, and a change of variables, we obtain
\begin{equation}\label{lemma W0 eq1_dom_conv}
\begin{split}
\int_a^{+\infty}G_s(x)\,ds
&=\int_a^{+\infty}s^{d-2}\frac{|sx|^{1-d/2}}{(2\pi)^{d/2}}\,K_{d/2-1}(|sx|)\,ds\\
&=\frac{|x|^{1-d}}{(2\pi)^{d/2}}\int_{a|x|}^{+\infty}t^{d/2-1}K_{d/2-1}(t)\,dt
=|x|^{1-d}W_a(x),
\end{split}
\end{equation}
which is \eqref{int Gs is W}. Regarding \eqref{deriv F is double deriv G}, by \eqref{formula daGa eq1}, \eqref{def Ga}, \eqref{def F}, and \eqref{def Fa}, we see that
\begin{equation}
\begin{split}
\frac{\partial}{\partial a}\Big(\frac{G_a(x)}{|x|^2}\Big)
=\frac{H_a(x)}{|x|}+\Big(\frac{d}{2}-1\Big)\frac{G_a(x)}{a|x|^2}
=\frac{a^{d-2}}{|x|}\Big(H(ax)+\Big(\frac{d}{2}-1\Big)\frac{G(ax)}{|ax|}\Big)
=-\frac{F_a(x)}{a^2|x|}.
\end{split}
\end{equation}

We now address the proofs of $(i)$ and $(ii)$. Regarding $(i)$, it is clear from 
\eqref{K int exp form} that $\kappa>0$. Moreover, by \eqref{K int exp form}, Fubini's theorem, and a change of variables we see that, for every $\alpha\geq0$,
\begin{equation}\label{lemma W0 eq1}
\begin{split}
\int_{0}^{+\infty}t^{\alpha}K_{\alpha}(t)&\,dt
=\int_{0}^{+\infty}\!\!\int_0^{+\infty} t^{\alpha} e^{-t\cosh r}\cosh(\alpha r)\,dr\,dt\\
&=\int_{0}^{+\infty}\frac{\cosh(\alpha r)}{(\cosh r)^{\alpha+1}}\int_0^{+\infty} 
s^{\alpha} e^{-s}\,ds\,dr
=\Gamma(\alpha+1)\int_{0}^{+\infty}\frac{\cosh(\alpha r)}{(\cosh r)^{\alpha+1}}\,dr,
\end{split}
\end{equation}
where $\Gamma$ denotes the Gamma function. Since $\Gamma(d/2)|\sph|=2\pi^{d/2}$ (see for example \cite[Proposition 0.7]{Folland}), \eqref{lemma W0 eq1} proves the equality in \eqref{W(0)}.
Observe also that, for $r>0$,
\begin{equation}\label{lemma W0 eq2}
\begin{split}
0<\frac{\cosh(\alpha r)}{(\cosh r)^{\alpha+1}}
=2^{\alpha}\frac{e^{\alpha r}+e^{-\alpha r}}{(e^r+e^{-r})^{\alpha+1}}
\leq2^{\alpha}\frac{2e^{\alpha r}}{e^{(\alpha+1)r}}=2^{\alpha+1}e^{-r},
\end{split}
\end{equation}
which is integrable in $(0,+\infty)$. Therefore, \eqref{lemma W0 eq1} and \eqref{lemma W0 eq2} show that $\kappa<+\infty$. With this in hand, and in view of the last equality in \eqref{lemma W0 eq1_dom_conv}, the fact that 
$\lim_{a\uparrow+\infty}W_a(x)=0$ and 
$\lim_{a\downarrow0}W_a(x)=\kappa$ follows by dominated convergence. The proof of $(i)$ is complete.

In order to prove $(ii)$ we need to study the asymptotic behavior of $K_\alpha(t)$ and $K_\alpha'(t)$ as $t\uparrow+\infty$ and as $t\downarrow 0$ when $\alpha\geq 0$. Let us first focus on the case $t\uparrow+\infty$.
By \cite[page 206]{Watson}, we know that
\begin{equation}\label{lemma W0 eq3}
K_\alpha(t)=\Big(\frac{\pi}{2t}\Big)^{1/2}\frac{e^{-t}}{\Gamma(\alpha+1/2)}
\int_0^{+\infty}e^{-r}r^{\alpha-1/2}\Big(1+\frac{r}{2t}\Big)^{\alpha-1/2}\,dr.
\end{equation}
Therefore, for $t>0$ big enough,
\begin{equation}\label{lemma W0 eq3 1}
K_\alpha(t)\leq\Big(\frac{\pi}{2t}\Big)^{1/2}\frac{2^{\alpha}e^{-t}}{\Gamma(\alpha+1/2)}\Big\{
\int_0^{1}\frac{dr}{\sqrt{r}}
+\int_1^{+\infty}e^{-r}r^{2\alpha}\,dr\Big\}.
\end{equation}
From \eqref{lemma W0 eq3 1} 
we deduce that for every $\alpha\geq0$ there exists $C_\alpha>0$ depending only on $\alpha$ such that 
\begin{equation}\label{lemma W0 eq3 3}
K_\alpha(t)\leq C_\alpha t^{-1/2}e^{-t}\quad\text{as }t\uparrow+\infty.
\end{equation}
Concerning $K'_\alpha$, note that
\begin{equation}
\cosh r\cosh(\alpha r)=\frac{1}{4}(e^r+e^{-r})(e^{\alpha r}+e^{-\alpha r})
=\frac{1}{2}\big(\cosh((\alpha+1) r)+\cosh(|\alpha-1| r)\big),
\end{equation}
thus using \eqref{K int exp form} we see that
\begin{equation}\label{lemma W0 eq3 4}
\begin{split}
K_{\alpha}'(t)=-\int_0^{+\infty} e^{-t\cosh r}\cosh r\cosh(\alpha r)\,dr
=-\frac{1}{2}\big(K_{\alpha+1}(t)+K_{|\alpha-1|}(t)\big)
\end{split}
\end{equation}
for all $t>0$ ---\eqref{lemma W0 eq3 4} also follows from \cite[9.6.26 and 9.6.6]{AbraStegun}. Then, \eqref{lemma W0 eq3 4} and \eqref{lemma W0 eq3 3} prove that
\begin{equation}\label{lemma W0 eq3 5}
|K_\alpha'(t)|\leq C_\alpha t^{-1/2}e^{-t}\quad\text{as }t\uparrow+\infty
\end{equation}
for some $C_\alpha>0$ depending only on $\alpha$.

Regarding the case $t\downarrow0$, from \cite[9.6.8 and 9.6.9]{AbraStegun} we know that $K_0(t)\sim-\ln t$ and  
$K_\alpha(t)\sim2^{\alpha-1}\Gamma(\alpha)\,t^{-\alpha}$ for $\alpha>0$ as $t\downarrow 0$.
Combining this estimate with \eqref{lemma W0 eq3 4}, we deduce that for every $\alpha\geq0$ there exists $C_\alpha>0$ depending only on $\alpha$ such that 
\begin{equation}\label{lemma W0 asympt_0}
|K_\alpha'(t)|\leq C_\alpha t^{-\alpha-1}\text{ \quad as $t\downarrow 0$}.
\end{equation}

Finally, since $K_\alpha'$ is continuous in $(0,+\infty)$, from \eqref{lemma W0 eq3 5} and \eqref{lemma W0 asympt_0} we conclude that for every $\alpha\geq0$ there exists $C_\alpha>0$ depending only on $\alpha$ such that 
\begin{equation}\label{K_alph_assimpt_est}
|K_\alpha'(t)|\leq C_\alpha \Big(t^{-\alpha-1}\chi_{(0,1)}(t)
+t^{-1/2}e^{-t}\chi_{[1,+\infty)}(t)\Big)\quad\text{for all $t>0$}.
\end{equation}

With the asymptotic behavior of $K_\alpha$ and $K_\alpha'$ in hand, we can now focus on the proof of~$(ii)$. Let us first address the first statement in $(ii)$. Take $x\in\R^d\setminus\{0\}$. From \eqref{formula daGa eq2} and using that $K_{d/2-1}$ is a positive function, we know that
\begin{equation}\label{lemma W0 eq3 6}
\frac{\partial}{\partial a}\Big(a^{3-d}\frac{\partial}{\partial a}
\Big(\frac{G_a(x)}{|x|^2}\Big)\Big)=a^{3-d}G_a(x)>0
\end{equation}
for all $a>0$. Additionally, using \eqref{formula daGa eq1}, \eqref{def Ga}, and 
\eqref{def G} we see that
\begin{equation}
\begin{split}
a^{3-d}\frac{\partial}{\partial a}\Big(\frac{G_a(x)}{|x|^2}\Big)
&=a\frac{H(ax)}{|x|}+\Big(\frac{d}{2}-1\Big)\frac{G(ax)}{|x|^2}\\
&=a^{2-d/2}\frac{|x|^{-d/2}}{(2\pi)^{d/2}}\,K_{d/2-1}'(a|x|)
+\Big(\frac{d}{2}-1\Big)a^{1-d/2}\frac{|x|^{-1-d/2}}{(2\pi)^{d/2}}\,K_{d/2-1}(a|x|)
\end{split}
\end{equation}
and, therefore, 
\begin{equation}\label{lemma W0 eq3 7}
\lim_{a\uparrow+\infty}a^{3-d}\frac{\partial}{\partial a}\Big(\frac{G_a(x)}{|x|^2}\Big)=0
\end{equation}
by \eqref{lemma W0 eq3 3} and \eqref{lemma W0 eq3 5}. 
In conclusion, \eqref{lemma W0 eq3 6} and \eqref{lemma W0 eq3 7} prove that 
$a^{3-d}\frac{\partial}{\partial a}(\frac{G_a(x)}{|x|^2})$, as a function of $a>0$, 
is strictly increasing and converges to $0$ as $a\uparrow+\infty$, thus 
\begin{equation}\label{lemma W0 eq3 8}
a^{3-d}\frac{\partial}{\partial a}\Big(\frac{G_a(x)}{|x|^2}\Big)<0
\end{equation}
for all $a>0$. Then, applying \eqref{formula daGa eq1} to \eqref{lemma W0 eq3 8}
and taking $a=1$, we obtain that
\begin{equation}
0>\frac{\partial}{\partial a}\Big(\frac{G_a(x)}{|x|^2}\Big)\biggr\rvert_{a=1}
=\frac{H(x)}{|x|}+\Big(\frac{d}{2}-1\Big)\frac{G(x)}{|x|^2}=-\frac{F(x)}{|x|},
\end{equation}
where we used \eqref{def F} in the last equality. Therefore, $F(x)>0$ for all $x\in\R^d\setminus\{0\}$.

Let us now address the other two statements in $(ii)$. Note that, since $d\geq2$ and $G>0$,
\begin{equation}
F(x)=-H(x)-\Big(\frac{d}{2}-1\Big)\frac{G(x)}{|x|}
\leq -H(x)
=-\frac{|x|^{1-d/2}}{(2\pi)^{d/2}}\,K_{d/2-1}'(|x|).
\end{equation} 
Then, \eqref{K_alph_assimpt_est} leads to
\begin{equation}
F(x)\leq C_d\Big(|x|^{1-d}\chi_{(0,1)}(|x|)
+|x|^{(1-d)/2}e^{-|x|}\chi_{[1,+\infty)}(|x|)\Big)\quad\text{for all $x\in\R^d\setminus\{0\}$},
\end{equation} 
for some $C_d>0$ depending only on $d$, as desired. This estimate easily yields $\int_{\R^d} F<+\infty$, which concludes the proof of $(ii)$. 
\end{proof}

\begin{proof}[Proof of \Cref{main thm}]
By 
\eqref{def Fa}, \eqref{def F}, and \eqref{def G}, we see that $\Phi(\Omega,a)$ defined in \eqref{defi Phi intro} can be rewritten as
\begin{equation}\label{Phi definition}
\begin{split}
\Phi(\Omega,a)&=
\int_{\paO}\!\int_{\paO}W_a(x-y)\frac{|\nu(x)-\nu(y)|^2}{|x-y|^{d-1}}
\,d\sigma(x)\,d\sigma(y)
+4\int_{\Omega^c}\!\int_{\Omega}\frac{F_a(x-y)}{|x-y|}\,dx\,dy.
\end{split}
\end{equation}
Then, thanks to \eqref{int Gs is W} and \eqref{deriv F is double deriv G}, we see that
\begin{equation}
\begin{split}
\frac{\partial}{\partial a}\Phi(\Omega,a)&=
-\int_{\paO}\!\int_{\paO}G_a(x-y)|\nu(x)-\nu(y)|^2
\,d\sigma(x)\,d\sigma(y)\\
&\quad-4\int_{\Omega^c}\!\int_{\Omega}\frac{\partial}{\partial a}
\Big(a^2\frac{\partial}{\partial a}
\Big(\frac{G_a(x-y)}{|x-y|^2}\Big)\Big)\,dx\,dy.
\end{split}
\end{equation}
From here, \eqref{monotony} follows directly from \Cref{thm fund ineq}, which also shows that the equality in \eqref{monotony} is attained for some (and thus for all) $a>0$ if and only if $\Omega\subset\R^d$ is a ball.

Next, we focus on the second identity in \eqref{limit zero_infty}. It will be proven by computing the limit as $a\uparrow+\infty$ of the two terms on the right hand side of \eqref{Phi definition} separately. Let us first address the limit of the second one. Given $R>0$, set $B_R:=\{x\in\R^d: |x|<R\}$.
Take $R$ big enough so that 
$\overline\Omega\subset B_{R/2}$.
Then, we can split
\begin{equation}\label{l limit split}
\begin{split}
\int_{\Omega^c}\!\int_{\Omega}\frac{F_a(x-y)}{|x-y|}\,dx\,dy
&=\int_{B_R}\!\int_{B_R}\!\!\!\frac{F_a(x-y)}{|x-y|}\,
\chi_{\Omega}(x)\chi_{\Omega^c}(y)\,dx\,dy
+\int_{B_R^c}\!\int_{\Omega}\frac{F_a(x-y)}{|x-y|}\,dx\,dy.
\end{split}
\end{equation}
In order to deal with the two terms on the right hand side of \eqref{l limit split},
recall that $F_a(x)=a^dF(ax)$ for $x\in\R^d\setminus\{0\}$ and that 
$F$ is a positive and radial function such that $0<\|F\|_{L^1(\R^d)}<+\infty$; see \Cref{lemma W F}$(ii)$ and $\eqref{def Fa}$. In particular, a change of variables gives
$\int_{\R^d}F_a(x)\,dx=\|F\|_{L^1(\R^d)}$ and, for every $\epsilon>0$, 
\begin{equation}\label{l limit app ident}
\lim_{a\uparrow+\infty}\int_{\{x\in\R^d:\,|x|>\epsilon\}}F_a(x)\,dx
=\lim_{a\uparrow+\infty}\int_{\{x\in\R^d:\,|x|>\epsilon a\}}F(x)\,dx=0.
\end{equation}
Therefore, $F_a/\|F\|_{L^1(\R^d)}$ converges, as 
$a\uparrow+\infty$, to the Dirac delta measure in the sense of distributions.

Concerning the first term on the right hand side of \eqref{l limit split}, since $F_a$ is radial, a symmetrization argument gives that
\begin{equation}
\begin{split}
\int_{B_R}\!\int_{B_R}\frac{F_a(x-y)}{|x-y|}\,
\chi_{\Omega}(x)\chi_{\Omega^c}(y)\,dx\,dy&=
\frac{1}{2}\int_{B_R}\!\int_{B_R}F_a(x-y)
\frac{|\chi_{\Omega}(x)-\chi_{\Omega}(y)|}{|x-y|}\,dx\,dy.
\end{split}
\end{equation}
Therefore, \cite[Theorem 1]{Davila} shows that
\begin{equation}\label{eq davila}
\lim_{a\uparrow+\infty}\int_{B_R}\!\int_{B_R}\frac{F_a(x-y)}{|x-y|}\,
\chi_{\Omega}(x)\chi_{\Omega^c}(y)\,dx\,dy
=C|\chi_{\Omega}|_{\operatorname{BV}(B_R)},
\end{equation}
where $C>0$ is some constant depending only on $d$ and
\begin{equation}
|\chi_{\Omega}|_{\operatorname{BV}(B_R)}
:=\sup\Big\{\int_{B_R}\chi_{\Omega}\,\ddiv\varphi:\,\varphi\in C^1_c(B_R,\R^d),\, |\varphi|\leq1\text{ in }B_R\Big\}.
\end{equation}
Since $|\chi_{\Omega}|_{\operatorname{BV}(B_R)}=\mathcal{P}(\Omega)$ because $\overline\Omega\subset B_R$, \eqref{eq davila} yields
\begin{equation}\label{l limit eq2}
\lim_{a\uparrow+\infty}\int_{B_R}\!\int_{B_R}\frac{F_a(x-y)}{|x-y|}\,
\chi_{\Omega}(x)\chi_{\Omega^c}(y)\,dx\,dy
=C\mathcal{P}(\Omega).
\end{equation}

Regarding the second term on the right hand side of \eqref{l limit split}, using that 
$\overline\Omega\subset B_{R/2}$, Fubini's theorem, and a change of variables we can easily estimate
\begin{equation}
\begin{split}
\int_{B_R^c}\!\int_{\Omega}\frac{F_a(x-y)}{|x-y|}\,dx\,dy
&\leq\int_{B_R^c}\!\int_{B_{R/2}}
\frac{F_a(x-y)}{|x-y|}\,dx\,dy\\
&\leq\frac{2}{R}\int_{B_{R/2}}\!\int_{\{y\in\R^d:\,|x-y|>R/2\}}
F_a(x-y)\,dy\,dx\\
&=\frac{2}{R}\,|B_{R/2}|\int_{\{y\in\R^d:\,|y|>R/2\}}F_a(y)\,dy,
\end{split}
\end{equation}
thus \eqref{l limit app ident} yields
\begin{equation}\label{l limit eq1}
\lim_{a\uparrow+\infty}\int_{B_R^c}\!\int_{\Omega}\frac{F_a(x-y)}{|x-y|}\,dx\,dy=0.
\end{equation}

I conclusion, plugging \eqref{l limit eq2} and \eqref{l limit eq1} in \eqref{l limit split}, we get
\begin{equation}\label{lim_Fa_infty}
\begin{split}
\lim_{a\uparrow+\infty}4\int_{\Omega^c}\!\int_{\Omega}\frac{F_a(x-y)}{|x-y|}\,dx\,dy=C\mathcal{P}(\Omega),
\end{split}
\end{equation}
where $C>0$ is some constant depending only on $d$. This is the limit as $a\uparrow+\infty$ of the second term on the right hand side of \eqref{Phi definition}.

Let us now compute the limit as $a\uparrow+\infty$ of the first term on the right hand side of \eqref{Phi definition}. By \eqref{def Fa}, \eqref{def F}, and \eqref{W(0)}, we have $0<W_a(x)<\kappa$ for all $a>0$ and 
$x\in\R^d\setminus\{0\}$ because $K_{d/2-1}$ is a positive function. This gives that
\begin{equation}
0\leq W_a(x-y)\frac{|\nu(x)-\nu(y)|^2}{|x-y|^{d-1}}\leq \kappa \frac{|\nu(x)-\nu(y)|^2}{|x-y|^{d-1}};
\end{equation}
note that the right hand side is integrable in 
$\paO\times\paO$ because we are assuming that 
$\mathcal{E}_2(\Omega)<+\infty$.
Therefore,
by \Cref{lemma W F}$(i)$ and dominated convergence we get
\begin{equation}\label{l limit eq3}
\begin{split}
\lim_{a\uparrow+\infty}\int_{\paO}\!\int_{\paO}W_a(x-y)\frac{|\nu(x)-\nu(y)|^2}{|x-y|^{d-1}}
\,d\sigma(x)\,d\sigma(y)=0.
\end{split}
\end{equation}

Finally, plugging \eqref{lim_Fa_infty} and \eqref{l limit eq3} in \eqref{Phi definition}, we conclude that
\begin{equation}\label{l limit eq4}
\begin{split}
\lim_{a\uparrow+\infty}\Phi(\Omega,a)=C(d)\mathcal{P}(\Omega)
\end{split}
\end{equation}
for some constant $C(d)>0$ depending only on $d$.
The precise value of $C(d)$ could be tracked from 
\cite[(4)]{Davila} and computing 
$\|F\|_{L^1(\R^d)}$. 
However, at the end of the proof we will show that  
$$C(d)=\kappa \frac{4 |\mathbb{S}^{d-2}|}{d-1}$$ (with 
${\kappa}$ as in \Cref{main thm}) by inspecting the function
$a\mapsto\Phi(\Omega,a)$ when $\Omega$ is the unit ball of $\R^d$, since $C(d)>0$ depends only on $d$. This will yield the second equality in \eqref{limit zero_infty}. 

Next, we focus on the first identity in \eqref{limit zero_infty}. It will be proven by computing the limit as $a\downarrow0$ of the two terms on the right hand side of \eqref{Phi definition} separately. Let us first address the limit of the second one.
From \Cref{lemma W F}$(ii)$ we see that
\begin{equation}\label{lim:a_0_F_eq_main_est}
\begin{split}
\frac{F_a(x)}{|x|}&\leq C_d\, a^d|x|^{-1}\Big(|ax|^{1-d}\chi_{(0,1)}(|ax|)
+|ax|^{(1-d)/2}e^{-|ax|}\chi_{[1,+\infty)}(|ax|)\Big)\\
&=C_d \Big(a|x|^{-d}\chi_{(0,\frac{1}{a})}(|x|)
+a^{(1+d)/2}|x|^{-(1+d)/2}e^{-a|x|}\chi_{[\frac{1}{a},+\infty)}(|x|)\Big).
\end{split}
\end{equation}
With this estimate in hand, we will show that 
\begin{equation}\label{lim:a_0_F}
\begin{split}
\lim_{a\downarrow0}\int_{\Omega^c}\!\int_{\Omega}\frac{F_a(x-y)}{|x-y|}\,dx\,dy=0.
\end{split}
\end{equation}
Let us first address the double integral of the first term on the right hand side of \eqref{lim:a_0_F_eq_main_est}. To this end, take $R>1$ big enough so that 
$\overline\Omega\subset B_{R/2}$. Then, we write
\begin{equation}\label{lim:a_0_F_eq2}
\begin{split}
a\int_{\Omega^c}\!\int_{\Omega}
\frac{\chi_{(0,\frac{1}{a})}(|x-y|)}{|x-y|^{d}}
\,dx\,dy
&=a\int_{B_R\setminus\Omega}\!\int_{\Omega}
\frac{\chi_{(0,\frac{1}{a})}(|x-y|)}{|x-y|^{d}}\,dx\,dy\\
&\quad+a\int_{B_R^c}\!\int_{\Omega}
\frac{\chi_{(0,\frac{1}{a})}(|x-y|)}{|x-y|^{d}}
\,dx\,dy
=:I_1(a)+I_2(a).
\end{split}
\end{equation}
Regarding $I_1(a)$, note that if $a<\frac{1}{2R}$ then $\chi_{(0,\frac{1}{a})}(|x-y|)=1$ for all $x\in\Omega\subset B_R$ and $y\in B_R\setminus\Omega$. Thus, since $t^{-d}\leq 2R t^{-d-a}$ for all $0<a\leq1$ and all $0<t\leq2R$, from \cite[Observation 6]{DFPV2013} we deduce that
\begin{equation}\label{lim:a_0_F_eq3}
\begin{split}
0\leq\lim_{a\downarrow0}I_1(a)= \lim_{a\downarrow0}a\int_{B_R\setminus\Omega}\!\int_{\Omega}
\frac{1}{|x-y|^{d}}\,dx\,dy
\leq \lim_{a\downarrow0}a\int_{B_R\setminus\Omega}\!\int_{\Omega}
\frac{2R}{|x-y|^{d+a}}\,dx\,dy=0.
\end{split}
\end{equation}
Concerning $I_2(a)$, using Fubini's theorem and that $\overline\Omega\subset B_{R/2}$, we can estimate
\begin{equation}\label{lim:a_0_F_eq4}
\begin{split}
I_2(a)&\leq a\int_{\Omega}
\int_{\{y\in\R^d:\,\frac{R}{2}<|x-y|<\frac{1}{a}\}}
\frac{1}{|x-y|^{d}}\,dy\,dx
=a|\Omega|
\int_{\{z\in\R^d:\,\frac{R}{2}<|z|<\frac{1}{a}\}}
\frac{dz}{|z|^{d}}\\
&=a|\Omega||\mathbb{S}^{d-1}|
\int_{\frac{R}{2}}^{\frac{1}{a}}\frac{dr}{r}
=|\Omega||\mathbb{S}^{d-1}|\, a \big(\ln({\textstyle \frac{1}{a}})-\ln({\textstyle \frac{R}{2}})\big),
\end{split}
\end{equation}
which shows that $\lim_{a\downarrow0}I_2(a)=0$. This combined with \eqref{lim:a_0_F_eq3} and \eqref{lim:a_0_F_eq2} yields
\begin{equation}\label{lim:a_0_F_eq5}
\begin{split}
\lim_{a\downarrow0} a\int_{\Omega^c}\!\int_{\Omega}
\frac{\chi_{(0,\frac{1}{a})}(|x-y|)}{|x-y|^{d}}
\,dx\,dy=0.
\end{split}
\end{equation}

Let us now address the double integral of the second term on the right hand side of \eqref{lim:a_0_F_eq_main_est}. Using Fubini's theorem and a change of variables we see that
\begin{equation}
\begin{split}
a^{(1+d)/2}&\int_{\Omega^c}\!\int_{\Omega}
\frac{e^{-a|x-y|}}{|x-y|^{(1+d)/2}}
\chi_{[\frac{1}{a},+\infty)}(|x-y|)\,dx\,dy\\
&\leq a^{(1+d)/2}\int_{\Omega}\!\int_{\R^d}
\frac{e^{-a|x-y|}}{|x-y|^{(1+d)/2}}
\chi_{[\frac{1}{a},+\infty)}(|x-y|)\,dy\,dx\\
&=a^{(1+d)/2}|\Omega|\int_{\R^d}
\frac{e^{-a|z|}}{|z|^{(1+d)/2}}
\chi_{[\frac{1}{a},+\infty)}(|z|)\,dz\\
&=a^{(1+d)/2}|\Omega||\mathbb{S}^{d-1}|\int_{\frac{1}{a}}^{+\infty}r^{(d-3)/2}e^{-ar}\,dr
=a|\Omega||\mathbb{S}^{d-1}|\int_{1}^{+\infty}s^{(d-3)/2}e^{-s}\,ds. 
\end{split}
\end{equation}
Therefore, 
\begin{equation}\label{lim:a_0_F_eq1}
\begin{split}
\lim_{a\downarrow0}a^{(1+d)/2}&\int_{\Omega^c}\!\int_{\Omega}
\frac{e^{-a|x-y|}}{|x-y|^{(1+d)/2}}
\chi_{[\frac{1}{a},+\infty)}(|x-y|)\,dx\,dy=0. 
\end{split}
\end{equation}
Now, \eqref{lim:a_0_F} is a consequence of \eqref{lim:a_0_F_eq_main_est}, \eqref{lim:a_0_F_eq5}, and \eqref{lim:a_0_F_eq1}.

Let us now compute the limit as $a\downarrow0$ of the first term on the right hand side of \eqref{Phi definition}.
Since we are assuming that $\mathcal{E}_2(\Omega)<+\infty$, 
using that $0<W_a(x)<\kappa$ for all $a>0$ and 
$x\in\R^d\setminus\{0\}$, \Cref{lemma W F}$(i)$, and dominated convergence we see that
\begin{equation}\label{limit eq3}
\begin{split}
\lim_{a\downarrow0}
\int_{\paO}\!\int_{\paO}&W_a(x-y)\frac{|\nu(x)-\nu(y)|^2}{|x-y|^{d-1}}
\,d\sigma(x)\,d\sigma(y)\\
&=\kappa \int_{\paO}\!\int_{\paO}\frac{|\nu(x)-\nu(y)|^2}{|x-y|^{d-1}}
\,d\sigma(x)\,d\sigma(y)=\kappa\, \mathcal{E}_2(\Omega).
\end{split}
\end{equation}

In conclusion, the first equality in \eqref{limit zero_infty} is a consequence of \eqref{Phi definition}, \eqref{lim:a_0_F}, and \eqref{limit eq3}.

Up to now we have proven that, if $\Omega\subset\R^d$ is a bounded open set with Lipschitz boundary such that $\mathcal{E}_2(\Omega)<+\infty$, then
\begin{eqnarray*}
&\frac{\partial}{\partial a}\Phi(\Omega,a)
\leq 0\quad\text{for all }a>0, \\
&\lim_{a\downarrow0}\Phi(\Omega,a)
=\kappa\,\mathcal{E}_2(\Omega),
\quad 
\lim_{a\uparrow+\infty}\Phi(\Omega,a)
=C(d)\mathcal{P}(\Omega),
\end{eqnarray*}
and $\frac{\partial}{\partial a}\Phi(\Omega,a)=0$ holds for some (and thus for all) $a>0$ if and only if $\Omega$ is a ball. Recall that, here, $C(d)$ is a constant depending only on $d$. Therefore, to prove the proposition, it only remains to show that 
\begin{equation}\label{eq:formula_Cd}
C(d)=\kappa \frac{4 |\mathbb{S}^{d-2}|}{d-1}.
\end{equation} 
This will follow by applying the previous results to the unit ball 
$\mathbb{B}^d:=\{x\in\R^d:\,|x|<1\}$. Since 
$\frac{\partial}{\partial a}\Phi(\mathbb{B}^d,a)=0$ for all $a>0$,
the function $a\mapsto \Phi(\mathbb{B}^d,a)$ is constant in $a$ and, therefore,
\begin{equation}
\begin{split}
C(d)|\mathbb{S}^{d-1}|&=C(d)\mathcal{P}(\mathbb{B}^d)
=\lim_{a\uparrow+\infty}\Phi(\mathbb{B}^d,a)=
\lim_{a\downarrow0}\Phi(\mathbb{B}^d,a)
=\kappa\,\mathcal{E}_2(\mathbb{B}^d)\\
&=\kappa \int_{\mathbb{S}^{d-1}}\!\int_{\mathbb{S}^{d-1}}|x-y|^{3-d}\,d\sigma(x)\,d\sigma(y)
=\kappa|\mathbb{S}^{d-1}|\int_{\sph}|x-e|^{3-d}\,d\sigma(x),
\end{split}
\end{equation}
where $e\in\sph$ is any unit vector ---we used the invariance of $\sph$ under rotations in the last equality above. This shows that 
$C(d)=\kappa\int_{\sph}|x-e|^{3-d}\,d\sigma(x)$
and means that, to prove \eqref{eq:formula_Cd}, we only need to show that 
\begin{equation}\label{eq:formula_Cd2}
\begin{split}
\int_{\sph}|x-e|^{3-d}\,d\sigma(x)=\frac{4 |\mathbb{S}^{d-2}|}{d-1}.
\end{split}
\end{equation}
To prove this, we write $x=(x_1,\ldots,x_d)\in\mathbb{S}^{d-1}\subset\R^d$ and we choose $e=(0,\ldots,0,1)$. 
Let $\theta \in [0, \pi]$ be the angle between $x$ and $e$, so that $x \cdot e = \cos\theta$. Then, 
$|x-e|^2=|x|^2+|e|^2-2x\cdot e=2(1-\cos\theta)=4\sin^2(\frac{\theta}{2}),$ that is, $|x-e|=2\sin\frac{\theta}{2}.$

We now slice the sphere $\mathbb{S}^{d-1}$ along the polar angle $\theta$. In this way, the surface measure $\sigma$ on $\mathbb{S}^{d-1}$ decomposes into the product of the 
$(d-2)$-dimensional surface measure of a cross-sectional sphere of radius $\sin\theta$ and the arc length element along the polar angle. Therefore, denoting by $|\mathbb{S}^{d-2}|$ the surface area of the unit sphere $\mathbb{S}^{d-2}\subset\mathbb{R}^{d-1}$ (with the convention $|\mathbb{S}^{0}|=2$ for $d=2$), we see that
\begin{equation}\label{eq:formula_Cd2_eq2}
\begin{split}
\int_{\sph}|x-e|^{3-d}\,d\sigma(x)
&=|\mathbb{S}^{d-2}|\int_{0}^{\pi}(2\sin\tfrac{\theta}{2})^{3-d} (\sin\theta)^{d-2} \, d\theta\\
&=|\mathbb{S}^{d-2}|\int_{0}^{\pi}(2\sin\tfrac{\theta}{2})^{3-d} (2\sin\tfrac{\theta}{2}\cos\tfrac{\theta}{2})^{d-2} \, d\theta\\
&=2|\mathbb{S}^{d-2}|\int_{0}^{\pi} (\cos\tfrac{\theta}{2})^{d-2} \sin\tfrac{\theta}{2}\, d\theta
=4|\mathbb{S}^{d-2}|\int_{0}^{1} t^{d-2} \,dt
=\frac{4|\mathbb{S}^{d-2}|}{d-1},
\end{split}
\end{equation}
which is \eqref{eq:formula_Cd2}. The proof of the proposition is now complete.
\end{proof}

\begin{proof}[Proof of \eqref{ineq norm1 main1} in \Cref{main thm intro}]
We can assume that $\mathcal{E}_2(\Omega)<+\infty$, otherwise the result is obvious, since $\mathcal{P}(\Omega)<+\infty$ because $\Omega$ is a bounded open set with Lipschitz boundary. Under this assumption, 
\eqref{ineq norm1 main1} follows directly from \Cref{main thm}. On the one hand we know that,
as a function of $a\in(0,+\infty)$,
$\Phi(\Omega,a)$ is constant if $\Omega$ is a ball and 
$\Phi(\Omega,a)$ is strictly decreasing if $\Omega$ is not a ball. On the other hand,
$\lim_{a\downarrow0}\Phi(\Omega,a)/\kappa$ and 
$\lim_{a\uparrow+\infty}\Phi(\Omega,a)/\kappa$ are equal to the left and right hand sides of \eqref{ineq norm1 main1}, respectively. The combination of these two facts yields  \eqref{ineq norm1 main1}, as well as the case of equality.
\end{proof}

\begin{proof}[Proof of \eqref{ineq norm1} in \Cref{main thm intro}]
Given $\epsilon>0$, set
\begin{equation}
\begin{split}
I_\epsilon(\Omega):=-\int_\paO\!\int_{\paO}\chi_{[\epsilon,+\infty)}(|x-y|)\frac{(y-x)\cdot\nu(x)}{|x-y|^d}\,d\sigma(x)\,d\sigma(y).
\end{split}
\end{equation}
Our goal is, on the one hand, to show that $I_\epsilon(\Omega)\leq\frac{1}{2}\,\mathcal{E}_1(\Omega)$ for all $\epsilon>0$ and, on the other hand, to show that $\lim_{\epsilon\downarrow0}I_\epsilon(\Omega)=\frac{1}{2}\,|\sph|\mathcal{P}(\Omega)$. From this, \eqref{ineq norm1}  follows. 

Let us first prove that $I_\epsilon(\Omega)\leq\frac{1}{2}\,\mathcal{E}_1(\Omega)$. Note that, from Cauchy-Schwarz inequality, 
\begin{equation}\label{ineq norm1_eq2}
(y-x)\cdot\big(\nu(y)-\nu(x)\big)\leq
|x-y||\nu(x)-\nu(y)|\quad\text{for $\sigma$-a.e. $x,y\in\paO$}
\end{equation}
and, assuming that $x\neq y$, the equality holds if and only if
\begin{equation}\label{ineq norm1_eq2:eq_case}
\nu(x)-\nu(y)=\frac{|\nu(x)-\nu(y)|}{|x-y|}(x-y).
\end{equation}
Then, if we symmetrize the double integral $I_\epsilon(\Omega)$ and we apply \eqref{ineq norm1_eq2} to the integrand, we deduce that
\begin{equation}\label{ineq norm1 eq1}
\begin{split}
I_\epsilon(\Omega)
&=\frac{1}{2}\int_\paO\!\int_{\paO}\chi_{[\epsilon,+\infty)}(|x-y|)
\frac{y-x}{|x-y|^d}\cdot\big(\nu(y)-\nu(x)\big)\,d\sigma(x)\,d\sigma(y)\\
&\leq\frac{1}{2}\int_\paO\!\int_{\paO}
\chi_{[\epsilon,+\infty)}(|x-y|)
\frac{|\nu(x)-\nu(y)|}{|x-y|^{d-1}}\,d\sigma(x)\,d\sigma(y)\leq\frac{1}{2}\,\mathcal{E}_1(\Omega).
\end{split}
\end{equation}

Next, we will prove that $\lim_{\epsilon\downarrow0}I_\epsilon(\Omega)=\frac{1}{2}\,|\sph|\mathcal{P}(\Omega)$. Since $\partial\Omega$ is Lipschitz, there exists $N\subset\partial\Omega$ such that $\sigma(N)=0$  
and $\partial\Omega$ has a well-defined tangent plane at all 
$y\in\partial\Omega\setminus N$. Given $y\in\partial\Omega\setminus N$, set 
$\Omega_{y,\epsilon}:=\{x\in\Omega:\,|x-y|>\epsilon\}$. Since 
\begin{equation}
\partial\Omega_{y,\epsilon}=\{x\in\paO:\,|x-y|\geq\epsilon\}\cup\{x\in\Omega:\,|x-y|=\epsilon\},
\end{equation}
if we also denote by $\sigma$ and $\nu$ the  surface measure and the outward unit normal vector field on 
$\partial\Omega_{y,\epsilon}$, respectively, by the Gauss-Green theorem we see that 
\begin{equation}
\begin{split}
\int_{\paO\cap\{x\in\R^d:\,|x-y|\geq\epsilon\}}&\frac{(y-x)\cdot\nu(x)}{|x-y|^d}\,d\sigma(x)
+\int_{\Omega\cap\{x\in\R^d:\,|x-y|=\epsilon\}}\frac{(y-x)\cdot\nu(x)}{|x-y|^d}\,d\sigma(x)\\
&=\int_{\partial\Omega_{y,\epsilon}}\frac{(y-x)\cdot\nu(x)}{|x-y|^d}\,d\sigma(x)
=\int_{\Omega_{y,\epsilon}}\ddiv_x\Big(\frac{y-x}{|x-y|^d}\Big)\,dx=0,
\end{split}
\end{equation}
where we used in the last equality that 
$\operatorname{dist}(y,\overline{\Omega_{y,\epsilon}})>0$ and, therefore, the integrand in the last integral vanishes identically in $\Omega_{y,\epsilon}$. Note that
$\nu(x)=\frac{y-x}{|x-y|}$ on $\{x\in\Omega:\,|x-y|=\epsilon\}$.
Thus,
\begin{equation}
\begin{split}
f_\epsilon(y)&:=-\int_{\paO\cap\{x\in\R^d:\,|x-y|\geq\epsilon\}}\frac{(y-x)\cdot\nu(x)}{|x-y|^d}\,d\sigma(x)
=\int_{\Omega\cap\{x\in\R^d:\,|x-y|=\epsilon\}}\frac{(y-x)\cdot\nu(x)}{|x-y|^d}\,d\sigma(x)\\
&=\frac{1}{\epsilon^{d-1}}
\sigma\big(\{x\in\Omega:\,|x-y|=\epsilon\}\big).
\end{split}
\end{equation}
From the last equality we see that $0\leq f_\epsilon(y)\leq|\sph|$ for all $\epsilon>0$. Moreover,
since $y\in\partial\Omega\setminus N$, we have $\lim_{\epsilon\downarrow0}f_\epsilon(y)=\frac{1}{2}|\sph|$. Therefore, by dominated convergence we conclude that
\begin{equation}\label{ineq norm1 eq2}
\begin{split}
\lim_{\epsilon\downarrow0}I_\epsilon(\Omega)
=\lim_{\epsilon\downarrow0}\int_{\paO\setminus N} f_\epsilon(y)\,d\sigma(y)
=\frac{1}{2}|\sph|\mathcal{P}(\Omega).
\end{split}
\end{equation}

Let us now address the case of equality in \eqref{ineq norm1}.
It is a simple exercise to show that if $\Omega$ is a ball then equality in \eqref{ineq norm1} holds.
Indeed, since the inequality is scale and translation invariant, we can assume that $\Omega=\mathbb{B}^d$ is the unit ball of $\R^d$. Then, one can explicitly compute (in the same way as we proved \eqref{eq:formula_Cd2})
\begin{equation}
\begin{split}
\mathcal{E}_1(\mathbb{B}^d) &= \int_{\sph}\!\int_{\sph}\frac{d\sigma(x)\,d\sigma(y)}{|x-y|^{d-2}}
=|\sph|\int_{\sph}\frac{d\sigma(x)}{|x-e|^{d-2}}
=|\sph|^2=|\sph|\mathcal{P}(\mathbb{B}^d),
\end{split}
\end{equation}
where here $e\in\sph$ is any unit vector. Alternatively, if $\Omega=\mathbb{B}^d$ then the equality in \eqref{ineq norm1_eq2} holds for all $x,y\in\sph$. In view of \eqref{ineq norm1 eq2} and \eqref{ineq norm1 eq1}, this leads to 
\begin{equation}
\frac{1}{2}|\sph|\mathcal{P}(\mathbb{B}^d)=\lim_{\epsilon\downarrow 0} I_\epsilon(\mathbb{B}^d)
=\lim_{\epsilon\downarrow 0}\frac{1}{2}\int_\sph\!\int_{\sph}
\chi_{[\epsilon,+\infty)}(|x-y|)
\frac{d\sigma(x)\,d\sigma(y)}{|x-y|^{d-2}}=\frac{1}{2}\,\mathcal{E}_1(\mathbb{B}^d),
\end{equation}
where we used monotone convergence in the last equality.

Assume now that $\Omega$ is a bounded open set with Lipschitz boundary for which equality in \eqref{ineq norm1} holds. In particular, this means that $\mathcal{E}_1(\Omega)<+\infty$. Therefore,
\begin{equation}\label{eq:equality_from_lim_to_all_epsilon1}
\Big|\chi_{[\epsilon,+\infty)}(|x-y|)\frac{y-x}{|x-y|^d}\cdot\big(\nu(y)-\nu(x)\big)\Big|
\leq\frac{|\nu(x)-\nu(y)|}{|x-y|^{d-1}}
\end{equation}
for all $\epsilon>0$ and the right hand side is integrable in $\paO\times\paO$. Then, from the first identity in \eqref{ineq norm1 eq1}, \eqref{ineq norm1 eq2}, and using that the equality in \eqref{ineq norm1} holds we deduce that
\begin{equation}\label{eq:equality_from_lim_to_all_epsilon2}
\begin{split}
\int_\paO\!\int_{\paO}&
\frac{y-x}{|x-y|^d}\cdot\big(\nu(y)-\nu(x)\big)\,d\sigma(x)\,d\sigma(y)\\
&=\lim_{\epsilon\downarrow0}\int_\paO\!\int_{\paO}
\chi_{[\epsilon,+\infty)}(|x-y|)
\frac{y-x}{|x-y|^d}\cdot\big(\nu(y)-\nu(x)\big)\,d\sigma(x)\,d\sigma(y)\\
&=\lim_{\epsilon\downarrow0}2I_\epsilon(\Omega)
=|\sph|\mathcal{P}(\Omega)
=\mathcal{E}_1(\Omega)
=\int_\paO\!\int_{\paO}
\frac{|\nu(x)-\nu(y)|}{|x-y|^{d-1}}\,d\sigma(x)\,d\sigma(y),
\end{split}
\end{equation}
where we used \eqref{eq:equality_from_lim_to_all_epsilon1} and dominated convergence in the first equality. Now, from this identity, and using that $\mathcal{E}_1(\Omega)<+\infty$, \eqref{ineq norm1_eq2}, and \eqref{ineq norm1_eq2:eq_case} we conclude that there exists $N\subset\paO$ with $\sigma(N)=0$ such that for all $y\in \paO\setminus N$ it holds that
\begin{equation}\label{eq:nu_terms_K}
\nu(x)-\nu(y)=\frac{|\nu(x)-\nu(y)|}{|x-y|}(x-y)
=k(x,y)(x-y)
\quad\text{for $\sigma$-a.e. }x\in\paO,
\end{equation}
where we denoted 
\begin{equation}\label{def:k_quotient}
k(x,y):=\frac{|\nu(x)-\nu(y)|}{|x-y|}\geq 0.
\end{equation}
Since $|\nu|=1$, using \eqref{eq:nu_terms_K} we then see that
\begin{equation}
\begin{split}
1&=\nu(x)\cdot\nu(x)=\big(\nu(y)+k(x,y)(x-y)\big)\cdot \big(\nu(y)+k(x,y)(x-y)\big)\\
&=1+2k(x,y)\nu(y)\cdot(x-y)+k(x,y)^2|x-y|^2,
\end{split}
\end{equation}
that is,
\begin{equation}\label{eq:K_identity_eq1}
\begin{split}
k(x,y)^2=-2k(x,y)\frac{\nu(y)\cdot(x-y)}{|x-y|^2}.
\end{split}
\end{equation}

We claim that if $y\in\paO\setminus N$ then $k(x,y)>0$ for $\sigma$-a.e. $x\in\paO$. Assuming this claim for a moment, if we divide by $k(x,y)$ on both sides of \eqref{eq:K_identity_eq1}, for all $y\in\paO\setminus N$ 
we get
\begin{equation}\label{eq:K_identity}
\begin{split}
k(x,y)=-2\frac{\nu(y)\cdot(x-y)}{|x-y|^2}
\quad\text{for $\sigma$-a.e. $x\in\paO$}.
\end{split}
\end{equation}
Finally, combining \eqref{eq:K_identity} with \eqref{eq:nu_terms_K}, we deduce that for all $y\in \paO\setminus N$ it holds that
\begin{equation}
\nu(x)=\nu(y)-2\,\frac{\nu(y)\cdot(x-y)}{|x-y|^2}\,(x-y)\quad\text{for $\sigma$-a.e. }x\in\paO.
\end{equation}
From this point and arguing as we did below \eqref{CS eq2} we conclude that $\Omega$ is a ball. Therefore, to finish the proof of the case of equality in \eqref{ineq norm1}, it only remains to prove the claim, which will be proven by contradiction below.

Let $y\in\paO\setminus N$ and
assume that there exists $X\subset\paO$ such that 
$\sigma(X)>0$ and $k(x,y)=0$ for all $x\in X$, where the function $k$ is defined in \eqref{def:k_quotient}. Setting 
$\nu_0:=\nu(y)$, this means that $\nu(x)=\nu_0$ for all $x\in X$. 
Suppose for a moment that, in fact, $\nu(x)=\nu_0$ for $\sigma$-a.e $x\in\paO$. 
Then, the Gauss-Green theorem would give
\begin{equation}
0=\int_\Omega \ddiv_x(\nu_0)\,dx
=\int_{\paO}\nu_0\cdot\nu(x)\,d\sigma(x)
=|\nu_0|^2\sigma(\paO)=\sigma(\paO)>0,
\end{equation}
which is a contradiction. Therefore, there exists $Z\subset\paO$ such that $\sigma(Z)>0$ and 
$\nu(z)\neq\nu_0$ for all $z\in Z$. In particular, since 
$\sigma(N)=0$, there exists $z\in Z\setminus N\neq\emptyset$ such that $\nu(z)\neq\nu_0$. Then, using \eqref{eq:nu_terms_K} on this particular $z\in\paO\setminus N$,
we get
\begin{equation}
\nu(x)-\nu(z)=k(x,z)(x-z)
\quad\text{for $\sigma$-a.e. }x\in X.
\end{equation}
Now, since $\nu(x)=\nu_0\neq\nu(z)$ for all $x\in X$, we see that $k(x,z)>0$ for all $x\in X$ and, thus,
\begin{equation}
x=z+\frac{1}{k(x,z)}(\nu_0-\nu(z))
\quad\text{for $\sigma$-a.e. }x\in X.
\end{equation}
This means that, if we denote $L:=\{z+t(\nu_0-\nu(z))\in\R^d:\,t\in\R\}$ (the straight line passing through $z$ in the direction $\nu_0-\nu(z)$), then $\sigma (X\setminus L)=0$. For $d\geq3$ this easily leads to the contradiction
\begin{equation}
0<\sigma(X)=\sigma (X\setminus L)+\sigma (X\cap L)\leq\sigma (L)=0,
\end{equation}
since $L$ is a 1-dimensional set and $\sigma$ is a $(d-1)$-dimensional measure. 
It only remains to address the case $d=2$. Since $\sigma (X\setminus L)=0$, we deduce that $\nu(x)\cdot(\nu_0-\nu(z))=0$ for $\sigma$-a.e. $x\in X$ ---that is, $\paO$ must be tangent to $L$ a.e. in $X$. Now, the assumption 
$\sigma(X)>0$ guarantees the existence of $x\in X$ such that 
$\nu(x)\cdot(\nu_0-\nu(z))=0$. Using that $\nu(x)=\nu_0$ because $x\in X$, this is equivalent to $\nu_0\cdot\nu(z)=1$. Since the left hand side is a scalar product of unit vectors, this last identity leads to $\nu(z)=\nu_0$, but this contradicts our choice of $z$.
\end{proof}

\subsection*{Acknowledgement}
The author thanks Vladimir Lotoreichik for suggesting that the {\em Reflection Lemma} \cite[Lemma 5.3]{Brock2007} could be useful within this topic, and Matteo Cozzi and Gyula Csat\'o for pointing out the argument behind the proof of the inequality \eqref{ineq norm1} in the case of bounded $C^2$ domains.

\subsection*{AI Usage} Artificial Intelligence (Gemini Pro) has only been used to help prove \Cref{lm:jump_fct}.


\begin{thebibliography}{AHMTT}

\bibitem{AbraStegun} M. Abramowitz, I. Stegun, {\em Handbook of mathematical functions with formulas, graphs, and mathematical tables}, Washington, D.C., 1964.

\bibitem{B} J. Bennett, {\em Heat-flow monotonicity related to some inequalities in euclidean analysis}, Contemporary Mathematics 505 (2010), Amer. Math. Soc. ISBN 978-0-8218-4770-1.

\bibitem{Borsuk1947} K. Borsuk, {\em Sur la courbure totale des courbes ferm\'es}, Ann. Soc. Pol. Math. 20 (1947), 251--265.

\bibitem{Brezis2002} H. Brezis, {\em How to recognize constant functions. Connections with Sobolev spaces}, Uspekhi Mat. Nauk 57 (2002), no. 4 (346), 59--74;
translation in Russian Math. Surveys 57 (2002), no. 4, 693--708.

\bibitem{Brock2007} F. Brock, CHAPTER 1 -- {\em Rearrangements and Applications to Symmetry Problems in PDE}, Editor(s): M. Chipot, Handbook of Differential Equations: Stationary Partial Differential Equations,
North-Holland, Volume 4, 2007, 1--60.

\bibitem{Chen1971}
B.-Y. Chen, {\em On a theorem of {F}enchel-{B}orsuk-{W}illmore-{C}hern-{L}ashof}, Math. Ann. {194} (1971), 19--26.

\bibitem{Davila} J. D\'avila, {\em On an open question about functions of bounded variation}, Calc. Var. 15 (2002), 519--527.

\bibitem{DPW} J. Dávila, M. del Pino, J. Wei, {\em Nonlocal $s$-minimal surfaces and Lawson cones}, J. Differential Geom. 109(1), (2018), 111--175.

\bibitem{DFPV2013} S. Dipierro, A. Figalli, G. Palatucci, E. Valdinoci, {\em Asymptotics of the $s$-perimeter as $s\downarrow0$}, Discrete and Continuous Dynamical Systems 33(7) (2013), 2777--2790.

\bibitem{Dominguez2023} O. Dom\'inguez, S. Tikhonov,  {\em Function Spaces of Logarithmic Smoothness: Embeddings and Characterizations}, Memoirs of the American Mathematical Society, 282, no. 1393, (2023) American Mathematical Society.

\bibitem{Fary1949} I. F\'ary, {\em Sur la courbure totale d'une courbe gauche faisant un n\oe ud}, Bulletin de la Soc. Math. de France, Tome 77 (1949), 128--138.

\bibitem{Fenchel1929} W. Fenchel, {\em \"Uber Kr\"ummung und Windung geschlossener Raumkurven}, Math. Ann. 101(1) (1929), 238--252.

\bibitem{FFMMM} A. Figalli, N. Fusco, F. Maggi, V. Millot, M. Morini, {\em Isoperimetry and Stability Properties of Balls with Respect to Nonlocal Energies}, Commun. Math. Phys. 336, (2015), 441--507.

\bibitem{Folland} G. Folland, {\em Introduction to partial differential equations}, second edition, Princeton Univ. Press, 1995.

\bibitem{HMMPT} S. Hofmann, E. Marmolejo-Olea, M. Mitrea, S. P\'erez-Esteva, M. Taylor, {\em Hardy spaces, singular integrals and the geometry of euclidean domains of locally finite perimeter}, Geom. Funct. Anal. 19(3) (2009), 842--882. 

\bibitem{HMT2007} S. Hofmann, M. Mitrea, M. Taylor, {\em Geometric and transformational properties of Lipschitz domains, Semmes-Kenig-Toro domains, and other classes of finite perimeter domains}, J. Geom. Anal., 17 (2007), no. 4, 593--647.

\bibitem{Kolmogorov1926} A. N. Kolmogorov, G. A. Seliverstov, {\em Sur la convergence des s\'eries de Fourier}, Atti Accad. Naz. Lincei, 3 (1926),  307--310.

\bibitem{MS} V. Maz'ya, T. Shaposhnikova, {\em On the Bourgain, Brezis, and Mironescu theorem concerning limiting embeddings of fractional Sobolev spaces}, J. Funct. Anal. 195 (2002), 230--238.

\bibitem{Milnor1950}
J. Milnor, {\em On the Total Curvature of Knots},
Ann. of Math. (2) {52} (1950), no.~2, 248--257.

\bibitem{Minkowski1903}
H. Minkowski, {\em Volumen und Oberfl\"ache}, Math. Ann. 57(4) (1903), 447--495.

\bibitem{Plessner1926} A. I. Plessner, {\em Ueber Konvergenz von trigonometrischen Reihen}, J. Reine Angew. Math., 155 (1926), 15--25.

\bibitem{Teschl} G. Teschl, {\em Mathematical methods in quantum mechanics. With applications to Schrödinger operators.}, Graduate Studies in Mathematics, American Mathematical Society, Providence, 2014.

\bibitem{Watson} G. N. Watson, {\em A Treatise on the Theory of Bessel Functions}, Cambridge University Press, 1995.

\bibitem{Willmore1968}
T. J. Willmore, {\em Mean curvature of immersed surfaces},
An. \c Sti. Univ. ``Al. I. Cuza'' Ia\c si Sec\c t. I a Mat. (N.S.) 14 (1968), 99--103.

\end{thebibliography}
\end{document}